\documentclass[11pt]{amsart}
\usepackage[english]{babel}
\usepackage{amsmath}
\usepackage{amsfonts}
\usepackage{amssymb}
\usepackage{amsthm}
\usepackage{mathdots} 

\usepackage[english]{babel}
\usepackage{yfonts}
\usepackage[T1]{fontenc}
\usepackage[utf8x]{inputenc}
\usepackage{enumerate}
\usepackage{verbatim}
\usepackage{graphicx}
\usepackage{verbatim}
\usepackage{faktor}
\usepackage{xcolor}
\usepackage{xfrac}
\usepackage{tikz}
\usetikzlibrary{calc}
\usetikzlibrary{intersections}
\usepackage[all]{xy}
\usepackage{cancel}

\newcommand{\calA}{\mathcal A}

\newcommand{\calF}{\mathcal F}

\newcommand{\PSL}{{\rm PSL}}
\renewcommand{\r}{\rho}

\renewcommand{\P}{{\rm P}}

\newcommand{\ov}{\overline}

\newcommand{\Gr}{{\rm Gr}}

\newcommand{\SO}{{\rm SO}}

\renewcommand{\d}{\delta}

\newcommand{\g}{\gamma}

\renewcommand{\l}{\lambda}

\newcommand{\G}{\Gamma}
\newcommand{\K}{\mathbb K}
\newcommand{\Id}{{\rm Id}}
\newcommand{\R}{\mathbb R}
\newcommand{\Z}{\mathbb Z}
\renewcommand{\P}{\mathbb P}

\newcommand{\N}{\mathbb N}

\newcommand{\SL}{{\rm SL}}
\renewcommand{\l}{\lambda}

\newcommand{\Sp}{{\rm Sp}}

\renewcommand{\>}{\rangle}
\newcommand{\bord}{\partial}
\newcommand{\Hom}{{\rm Hom}}
\newcommand{\un}{\underline}

\newcommand{\dual}{\flat}

\newcommand{\dg}{\partial_{\infty}\Gamma}
\newcommand{\Pp}{\mathbb{P}}
\newcommand{\tv}{\hspace{1mm}\pitchfork\hspace{1mm}}
\newcommand{\ntv}{\hspace{1mm}\cancel{\pitchfork}\hspace{1mm}}

\newcommand{\pc}{{\rm pcr}}
\newcommand{\gc}{{\rm gcr}}

\newcommand{\bpm}{\begin{pmatrix}}
\newcommand{\epm}{\end{pmatrix}}

    \newcommand\quotient[2]{
        \mathchoice
            {
                \text{\raise.5ex\hbox{$#1$}\big/\lower.5ex\hbox{$#2$}}%
            }
            {
                      \text{\raise.25ex\hbox{$#1$}\big/\lower.25ex\hbox{$#2$}}%
            }
            {
                #1\,/\,#2
            }
            {
                #1\,/\,#2
            }
    }

\theoremstyle{plain}
\newtheorem{thm}{Theorem}[section]
\newtheorem{lem}[thm]{Lemma}
\newtheorem{prop}[thm]{Proposition}
\newtheorem{cor}[thm]{Corollary}

\newtheorem*{teo*}{Theorem}

\newtheorem{notation}[thm]{Notation}

\theoremstyle{definition}
\newtheorem{example}[thm]{Example}

\newtheorem{defn}[thm]{Definition}
\newtheorem{remark}[thm]{Remark}

\newcommand{\thismonth}{\ifcase\month 
  \or January\or February\or March\or April\or May\or June%
  \or July\or August\or September\or October\or November%
  \or December\fi}

\title[Collar Lemma]{A collar lemma for partially hyperconvex surface group representations}
\author{Jonas Beyrer and Beatrice Pozzetti}
\date{\today}
\begin{document}
\thanks{J.B. acknowledges funding by the Deutsche Forschungsgemeinschaft (DFG, German Research Foundation), 338644254 (SPP2026), and the Schweizerischer Nationalfonds (SNF, Swiss Research Foundation), P2ZHP2 184022 (Early Postdoc.Mobility). B.P acknowledges funding by the DFG, 427903332 (Emmy Noether). Both authors acknowledge funding by the DFG, 281869850 (RTG 2229). BP thanks Tengren Zhang and Andres Sambarino for inspiring discussions on the topic of this article, we thank Nicolas Tholozan for an interesting discussion and for sharing his preprint \cite{TholPreprint} with us. We thank the anonymous referee for carefully reading the paper and the valuable remarks that helped improving its exposition}

\begin{abstract}
We show that a collar lemma holds for Anosov representations of fundamental groups of surfaces into $\SL(n,\R)$ that satisfy partial hyperconvexity properties inspired from Labourie's work. This is the case for several open sets of Anosov representations not contained in higher rank Teichm\"uller spaces, as well as for  $\Theta$-positive representations into $\SO(p,q)$ if $p\geq  4$. We moreover show that 'positivity properties' known for Hitchin representations, such as being positively ratioed and having positive eigenvalue ratios, also hold for partially hyperconvex representations.
\end{abstract}

\maketitle
\tableofcontents

\section{Introduction} 
\addtocontents{toc}{\protect\setcounter{tocdepth}{1}}
A fundamental result on the geometry of hyperbolic surfaces is the collar lemma which states that, in a hyperbolic surface $S_g$, every simple closed geodesic $c$ admits an embedded collar neighbourhood whose length diverges logarithmically as the length of $c$ shrinks to zero \cite{Keen}. This phenomenon is at the basis of various compactness results for moduli spaces, and admits algebraic reformulations useful to study the length spectrum of a hyperbolization: if two elements $\gamma$ and $\delta$ in the fundamental group $\pi_1(S_g)$ intersect geometrically, then there is an explicit lower bound on the length of $\gamma$ which is a function only dependent of the length of $\delta$. 

Higher rank Teichm\"uller theories, which include Hitchin representations and maximal representations, form connected components of the character variety $\Hom(\pi_1(S_g),\sf G)/\!/\sf G$ that consist only of discrete and faithful representations. These components form a robust generalization of the Teichm\"uller component, which is the only component with such property for the group $\sf G=\SL(2,\R)$. A number of geometric features of holonomies of hyperbolizations have been generalized for such higher rank Teichm\"uller theories (after  minor algebraic reformulations): this is the case for Basmajian and McShane identities \cite{LabMc,VY, FP} and for the collar lemma, which was proven for Hitchin representations by Lee-Zhang \cite{LZ} and for maximal representations by Burger-P. \cite{BP}. It was conjectured that the validity of the collar lemma distinguishes higher rank Teichm\"uller theories within the larger class of Anosov representations, the by now acclaimed generalization of convex cocompactness to higher rank.

In this paper we show that, instead, a collar lemma holds also beyond higher rank Teichm\"uller theories, and we generalize it to other (conjectural) classes of higher rank Teichm\"uller theories. 
To be more precise we study Anosov representations $\rho:\G \to \SL(E)$ of fundamental groups of surfaces $\G=\pi_1(S)$ on a real vector space $E$ of dimension $d$. Given such a representation, and for every  element $g\in \G$, we denote   by $\lambda_1(\r(g)),\ldots, \lambda_d(\rho(g))$ the (generalized) eigenvalues of the matrix $\rho(g)$ ordered so that their absolute values are non-increasing. Moreover we call two elements $g,h\in\G=\pi_1(S_g)$ \emph{linked}, if the corresponding closed geodesics with respect to some (and thus any) hyperbolic metric intersect in $S_g$.

Under specific hyperconvexity assumptions inspired from Labourie's work, called property $H_k$ and property $C_k$, we show

\begin{thm}\label{thm.INTROcollar lemma}
Let $\r:\G\to \SL(E)$ be an Anosov representation satisfying properties $H_{k-1},H_{k}, H_{k+1},H_{d-k-1},H_{d-k}, H_{d-k+1}$ and $C_{k-1}$, $C_k$. Then for any linked pair $g,h\in \G$ it holds
\begin{align*}
\frac{\lambda_1\ldots\lambda_k}{\lambda_d\ldots\lambda_{d-k+1}}(\rho(g))>\left(1-\frac{\lambda_{k+1}}{\lambda_k}(\rho(h))\right)^{-1}.
\end{align*}
\end{thm}
Here $H_0, C_0$ and $H_d$ are empty conditions. We now introduce the hyperconvexity properties needed as assumptions in Theorem \ref{thm.INTROcollar lemma}, and some other consequences of these properties that we establish in the paper. 
\subsection*{Property $H_k$ and positively ratioed representations}
Recall that, for every $l\in\{1,\ldots, d-1\}$, a $l$-Anosov representation admits a continuous equivariant boundary map $\xi_l:\partial\G\to \Gr_l(E)$. 

Following Labourie \cite[Section 7.1.4]{Labourie-IM} we say
\begin{defn}
A representation $\rho:\G\to\SL(E)$ satisfies \emph{property $H_k$} if it is $\{k,d-k-1,d-k+1\}$-Anosov, and  for every pairwise distinct triple $x,y,z\in\partial\G$ the sum
$$(\xi_k(x)\cap\xi_{d-k+1}(z))+(\xi_k(y)\cap\xi_{d-k+1}(z))+\xi_{d-k-1}(z)$$
is direct. 
\end{defn}

In \cite{MZ} Martone and Zhang introduced the notion of positively ratioed representations: those are Anosov representations that satisfy some additional 'positivity' property ensuring that suitable associated length functions can be computed as intersection with a geodesic current. In the same paper they have shown that the most studied examples of representations in higher rank Teichm\"uller theories, i.e. Hitchin and maximal representations, satisfy this positivity. 

We add to this by showing that representations satisfying properties $H_k$, $H_{d-k}$ are also positively ratioed. This provides new examples of this notion and in particular the first open sets in $\Hom(\G,\SL(E))$ of positively ratioed representations that are not in higher Teichm\"uller spaces:

\begin{thm}\label{thm.INTROhyperconvex implies positively ratioed}
Let $\r:\G\to\SL(E)$ be a $\{k-1,k,k+1\}-$Anosov representation satisfying property $H_k$ and $H_{d-k}$. Then $\r$ is $k-$positively ratioed.
\end{thm}
It is also possible to deduce Theorem \ref{thm.INTROhyperconvex implies positively ratioed} following the lines of Labourie's proof for Hitchin representations \cite[Section 4.4]{Labourie-IHES} using that, whenever a representation $\rho$ has property $H_k$, the image of its associated boundary map is a $C^1$-circle in $\Gr_k(E)$ \cite[Proposition 8.11]{PSW}.
 The argument we provide here is, however, more direct and closer to the circle of ideas important in the rest of the paper.

Theorem \ref{thm.INTROhyperconvex implies positively ratioed} lets us add to Martone-Zhang's list of positively ratioed representations a few more representations in (conjectural) higher Teichm\"uller theories:
 Hitchin representations into $\SO(p,p)$ and 
$\Theta$-positive representations into $\rho:\G\to\SO(p,q)$ as introduced by Guichard-Wienhard \cite{GWpositivity}. A straightforward computation shows that $\rho$ satisfies property $H_k$ if and only if the dual $\r^{\dual}$ satisfies property $H_{d-k}$. Moreover representations into $\SO(p,q)$ are self dual. It was proven in  \cite[Theorem 9.9]{PSW} and \cite[Theorem 10.1]{PSWB} that representations in the Hitchin component in $\SO(p,p)$ and $\Theta$-positive representations in $\SO(p,q)$  satisfy property $H_k$. As a result we obtain:

\begin{cor}The following are examples of positively ratioed representations:
\begin{enumerate}
\item If $\rho:\G\to\SO(p,p)$ belongs to the Hitchin component, then it is $k$-positively ratioed for $1\leq k\leq p-2$, and both irreducible factors of $\wedge^p\rho$ are 1-positively ratioed.
\item If $\rho:\G\to\SO(p,q)$ is $\Theta-$positive in the sense of Guichard-Wienhard, then it is $k$-positively ratioed for $1\leq k\leq p-2$.
\end{enumerate}
\end{cor}

We now turn to the second feature of  representations satisfying property $H_k$, which could be of independent interest. This justifies why in Theorem \ref{thm.INTROcollar lemma} we do not need to take the absolute value:

\begin{prop}\label{prop.INTROroots are positive for Hk}
If $\r$ satisfies property $H_k$, then for every $h\in\G\backslash\{e\}$ we have $$\frac{\lambda_k}{ \lambda_{k+1}}(\rho(h))>1;$$ equivalently the $k-$th and $(k+1)-$th eigenvalue of $\rho(h)$ have the same sign.
\end{prop}
One can show that for a $k-$Anosov representation $\r:\G\to\SL(E)$ and every non-trivial $g\in\G$
$$\frac{\lambda_1\ldots\lambda_k}{\lambda_d\ldots\lambda_{d-k+1}}(\rho(g))>1.$$
In particular, if $\r:\G\to\SL(E)$ is a representation satisfying properties $H_1,\ldots,H_k$ and $H_{d-k+1},\ldots,H_d$. Then the signs of all $\lambda_j(\rho(g))$ for $g\in\G\backslash\{e\}$ and $j\in \{1,\ldots,k+1\}\cup \{d-k,\ldots,d\}$ are equal.

\subsection*{Property $C_k$ and convexity of projections}
In the paper we introduce and study a second hyperconvexity property of representations, property $C_k$: 
\begin{defn}
A representation $\rho:\G\to\SL(E)$ satisfies \emph{property $C_k$} if it is $\{k,{k+1},{d-k-2},\allowbreak {d-k+1}\}$-Anosov, and  for every pairwise distinct triple $x,y,z\in\partial\G$ the sum
$$\xi_{d-k-2}(x)+\left(\xi_k(y)\cap\xi_{d-k+1}(x)\right)+\xi_{k+1}(z)$$
is direct. 
\end{defn}
We prove that property $C_k$ together with property $H_k$ implies that the shadow $\xi_x$ of the $k$-curve $\xi_k:\partial \G\to \Gr_k(E)$ in the projective plane associated to the quotient $\P\left(\xi_{d-k+1}(x)/\xi_{d-k-2}(x)\right)$ is itself hyperconvex; this means that the sum $\xi_x(y)\oplus \xi_x(z)\oplus \xi_x(w)=\R^3$ for all pairwise distinct $y,z,w\in \bord \G$.
\begin{prop}\label{prop.INTRO.projectionhyperconvex}
If $\rho:\G\to \SL(E)$  satisfies property $H_k$ an $C_k$, then for every $x\in\bord \G$ the curve
$$\left\{\begin{array}{ccl}
y&\mapsto &[\xi^{k}(y)\cap \xi^{d-k+1}(x)]\\
x&\mapsto &[\xi^{d-k-1}(x)]\\
\end{array}\right.$$
is a continuous hyperconvex curve in the plane $\P(\xi^{d-k+1}(x)/\xi^{d-k-2}(x))$. 
\end{prop}

We say that a representation $\rho:\G\to \SL(E)$ is \emph{Fuchsian} if it is  obtained composing a representation of $\SL(2,\R)$ with the holonomy of a hyperbolization. It is easy to check which Fuchsian representations  satisfy property $C_k$: if we split $E=E_1\oplus\ldots\oplus E_l$ as a direct sum of irreducible $\SL(2,\R)$-modules of non-increasing dimensions, the induced representation has property $C_k$ if and only if $\dim E_1-\dim E_2\geq 2k+3$. Furthermore, we show that representations satisfying property $C_k$ form a union of connected components of strongly irreducible representations that are Anosov in the right degrees:
\begin{prop}\label{prop.INTROCkclopen} 
 Property $C_k$ is open and closed among strongly irreducible $\{k,k+1,d-k-2, d-k+1\}-$Anosov representations satisfying property $H_{k}$.
\end{prop}

\subsection*{Comparison to Lee-Zhang and higher rank Teichm\"uller theories}
In the case of Hitchin representations into $\SL_d(\R)$ Lee-Zhang \cite{LZ} proved a collar Lemma comparing $\frac{\lambda_k}{\lambda_{k+1}}(\rho(g))$ to $\frac{\lambda_1}{\lambda_d}(\rho(h))$. For $k\neq 1$ this is a stronger result than ours.\footnote{Lee-Zhang use in their proof strong properties of the Frenet curve associated to the Hitchin representations; a tool that we cannot use with our assumptions.} However in our generality, it is not to expect that $\frac{\lambda_1}{\lambda_d}(\rho(h))$ is well behaved (we do not assume that the  representations are $1-$Anosov). In particular under our assumptions comparing $k-$th root and $k-$th weight seems to be the natural choice.

Theorem \ref{thm.INTROcollar lemma} yields also new results for higher rank Teichm\"uller theories. Indeed we prove that Guichard-Wienhard's $\Theta$-positive representations into $\SO(p,q)$ satisfy property $C_k$:
\begin{prop}\label{p.INTROc1}
Let $\rho:\G\to \SO(p,q), p<q$ be $\Theta$-positive Anosov. For every $1\leq k\leq p-3$ 
the representation $\rho$ has property $C_k$.
\end{prop}

Thus Theorem \ref{thm.INTROcollar lemma} yields:

\begin{cor}
Let $\rho:\G\to \SO(p,q)$ be $\Theta$-positive Anosov. Then for any linked pair $g,h\in\G\setminus\{e\}$ and all $1\leq k \leq p-3$
\begin{align*}
\frac{\lambda_1\ldots\lambda_k}{\lambda_d\ldots\lambda_{d-k+1}}(\rho(g))>\left(1-\frac{\lambda_{k+1}}{\lambda_k}(\rho(h))\right)^{-1}.
\end{align*}
\end{cor}

 \subsection*{Geometric reformulations and counterexamples}
It is possible to give a more geometric reformulation of the collar lemma in terms of the naturally defined (pseudo) length functions 
$$\ell_{\omega_k+\omega_{d-k}}(\rho(g))=\log \left|\frac{\lambda_1\ldots\lambda_k}{\lambda_d\ldots\lambda_{d-k+1}}(\rho(g))\right|, \quad \ell_{\alpha_k}(\rho(h))=\left|\frac{\lambda_{k}}{\lambda_{k+1}}(\rho(h))\right|.$$
Here the first quantity corresponds to the translation length of $\rho(g)$ on the symmetric space endowed to the Finsler distance associated to the symmetrized $k$-weight. Instead the second quantity doesn't, in general, come from a metric on the symmetric space: for example, it doesn't satisfy the triangle inequality. On the other hand, $\ell_{\alpha_k}$ is, in many ways, a better generalization of the hyperbolic length function, at least for representation $\rho:\G\to\SL(E)$ satisfying property $H_k$: 
for example it is proven in \cite{PSW} that the associated entropy is constant and equal to one, 
and in \cite[Appendix A]{PSW} that the pressure metric associated to the first root  has, on the Hitchin component, more similarities to the Weyl-Petersson metric than the usual pressure metric.

Theorem \ref{thm.INTROcollar lemma} can be reformulated  in terms of these geometric quantities:
\begin{align*}
e^{\ell_{\omega_k+\omega_{d-k}}(\rho(g))}>\left(1-e^{-\ell_{\alpha_k}(\rho(h))}\right)^{-1}.
\end{align*}

Note that, since the eigenvalues $\lambda_i$ are ordered so that their modulus does not increase,  we have that $\ell_{\omega_k+\omega_{d-k}}(\rho(h))>\ell_{\alpha_k}(\rho(h))$ and thus
$$\left(1-e^{-\ell_{\alpha_k}(\rho(h))}\right)^{-1}>\left(1-e^{-\ell_{\omega_k+\omega_{d-k}}(\rho(h))}\right)^{-1}.$$

If one is only interested in the length function $\ell_{\omega_k+\omega_{d-k}}$, this yields the following version of the collar lemma.

\begin{cor}\label{cor.weight collar lemma}
If $\r:\G\to\SL(E)$ satisfies the assumptions of Theorem \ref{thm.INTROcollar lemma}, then
\begin{align*}
e^{\ell_{\omega_k+\omega_{d-k}}(\rho(g))}>\left(1-e^{-\ell_{\omega_k+\omega_{d-k}}(\rho(h))}\right)^{-1}.
\end{align*}
\end{cor}
After this work was completed we got to know that Nicolas Tholozan independently obtained Corollary \ref{cor.weight collar lemma} with different techniques \cite{TholPreprint}.

While it might not seem very natural at first sight to compare two different length functions for the collar lemma, we have good reasons to do so: on the one hand the collar lemma in Theorem \ref{thm.INTROcollar lemma} is stronger than the one in Corollary \ref{cor.weight collar lemma}.
On the other hand we prove that a 'strong' collar lemma, relating $\ell_{\alpha_k}(\rho(h))$ to $\ell_{\alpha_k}(\rho(g))$ for a linked pair $g,h\in\G$ cannot, in general, hold.  We construct sequences of positive representations $\rho_n:\G_{1,1}\to\PSL(3,\R)$ from the fundamental group of the once punctured torus for which the stronger statement fails:
 \begin{thm}\label{thm.INTROcounterexample}
There is a one parameter family of positive representations $\rho_x:\G_{1,1}\to \PSL(3,\R)$, for $x\in(0,\infty)$, and $\g$, $\d\in\G_{1,1}$ such that 
$$\ell_{\alpha_1}(\rho_x(\g))=\ell_{\alpha_1}(\rho_x(\delta))\to 0$$
as $x$ goes to zero.
\end{thm}
This ensures the existence of a sequence of Hitchin representations from $\pi_1(S_2)$ with the same properties.

\subsection*{Sketch of the proof}
The proof of the collar lemma is based on the comparison between two cross ratios which can be associated to the boundary map, a projective cross ratio that computes the eigenvalue gap $\lambda_{k}/\lambda_{k+1}(\rho(h))$, and a Grassmannian cross ratio which computes the left hand side  in the expression of Theorem \ref{thm.INTROcollar lemma}. Using property $H_k$, the standard transformation laws of the projective cross ratio lets us obtain an upper bound on the right hand side (this step follows the same lines as \cite{LZ}). Then the connection between the two cross ratios yields an upper bound in terms of a Grassmannian cross ratio, involving, as one of its four entries, the space $\left(\xi_{d-k+1}(h_-)\cap \xi_{k}(g_+)\right)\oplus \xi_{d-k-1}(h_-) $. The bulk of the proof consists in showing that replacing this last subspace with $\xi_{d-k}(g_+)$ only increases the cross ratio. This latter step is obtained by considering a natural Lipschitz path interpolating between the two $k$-dimensional subspaces. Since the representation has properties $H_{d-k+1}$ and $H_{d-k-1}$ such path is a monotone curve in a $C^1$-surface inside the $(d-k)-$Grassmannian, and the proof reduces to studying the horizontal and vertical derivatives. That's where the properties $C_k$ and $C_{k-1}$, as well as the additional $H_j$ properties come into play.

 \subsection*{Outline of the paper} 
In Section \ref{s.prel} we set few standing assumptions and recall basic facts about Anosov representations that will be needed in the paper. In Section \ref{s.cr} we introduce the two cross ratios that will play an important role in the paper, and find useful ways to relate them. Section \ref{s.hyp} is devoted to the study of the partial hyperconvexity properties, property $H_k$ and $C_k$. Here is where their basic properties are proven: the conditions are open and closed among irreducible, have important implications on projections. We also discuss validity of these properties for $\Theta$-positive representations and Fuchsian loci. In Section \ref{sec:posrat} we discuss positively ratioed representations, and prove Theorem \ref{thm.INTROhyperconvex implies positively ratioed}. In Section \ref{s.collar} we prove the collar lemma, Theorem \ref{thm.INTROcollar lemma}, and in Section \ref{sec.strong collar} we construct the counterexample of Theorem \ref{thm.INTROcounterexample}.
\addtocontents{toc}{\protect\setcounter{tocdepth}{2}}

\section{Preliminaries}\label{s.prel}

We begin with some conventions and notations that we keep for the rest of this paper.

\begin{notation} In the ongoing we have
\begin{itemize}
\item $E$ will always be a real vector space of dimension $d$
\item $\G$ always a surface group, i.e. $\G=\pi_1(S_g)$ for $S_g$ a closed surface of genus at least $g\geq 2$.
\end{itemize}
\end{notation}

Since $\G$ is a surface group, the Gromov boundary $\dg$ is homeomorphic to a circle.\footnote{Actually all our results equally well work for hyperbolic groups with circle boundary, i.e. virtual surface groups by \cite{Gabai}; one would only need to replace 'non-trivial element of $\G$' with 'infinite order element of $\G$'.} This induces an order on the boundary:

\begin{defn} We call a tuple $(x_1,\ldots,x_n)\in \dg^n$ of distinct points with $n\geq 4$ \emph{cyclically ordered} or \emph{in that cyclic order} if the points are in positive order on $\dg\simeq S^1$ for one of the two orientations.
\end{defn}

 \begin{figure}[h]
\begin{tikzpicture}[scale=.8]
\draw (0,0) circle [radius =1];
\node at (1,0) [right] {$x_4$};
\node at (-1,0) [left] {$x_1$};
\node at (0,1) [above] {$x_2$};
\node at (0,-1) [below] {$x_5$};
\node at (.8,.8)[right]{$x_3$};
\filldraw (.7,.7) circle [radius=1pt];
\filldraw (1,0) circle [radius=1pt];
\filldraw (-1,0) circle [radius=1pt];
\filldraw (0,1) circle [radius=1pt];
\filldraw (0,-1) circle [radius=1pt];
\end{tikzpicture}
\caption{Cyclically ordered $5-$tuples $(x_1,\ldots,x_5), (x_5,\ldots,x_1)$}
\end{figure}
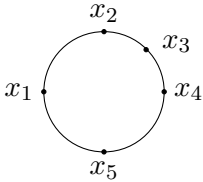

Note that such a tuple will never be in positive order for both orientations. Moreover with our convention every cyclic shift of a cyclically ordered tuple is still cyclically ordered, and the same holds if the order of the tuple is reversed.
As in \cite{MZ} we use the following notation for intervals:
\begin{notation}\label{n.interval}
Given pairwise distinct points $x,y,z\in\dg$ we denote by $(x,y)_z$ the connected component of $\dg\backslash \{x,y\}$ that does \emph{not} contain $z$.
\end{notation}

Throughout the paper we will be concerned with various subspaces $V<E$; it will  often be convenient to record the  dimension of such subspaces with an exponent, so that the notation $V^v$ will imply that $\dim(V)=v$. If $V^v,W^w ,U^u$ are subspaces of  $E$ we will write $V^v\oplus W^w=U^u$ to mean that the sum of $V^v$ and $W^w$ as subspaces of $U^u$ is direct and spans the subspace $U^u$, thus in particular $u=v+w$.

We denote by $\Gr_k(E)$ the Grassmannian of $k$-planes in $E$. Given $V\in\Gr_k(E)$ and $W\in\Gr_{d-k}(E)$, we write $V\tv W$ if $V$ and $W$ are \emph{transverse}, i.e. $V\oplus W=E$, and $V\ntv W$ if they are not transverse.

\begin{notation}\label{n.1}
When two nested subspaces $X^i<X^j$ are fixed, with quotient  $X:=\quotient{X^j}{ X^i}$, we will denote by  $[\cdot]_X$ the projection that associates to a subspace $V^l\in \Gr_l(X^j)\setminus \Gr_l(X^i)$ the subspace of $X$ of dimension  $l':=l-\dim (V^l\cap X^i)$ given by 
$$[V^l]_X:= \quotient{(V^l+X^i)}{ X^i}\in \Gr_{l'}(X).$$
\end{notation}

Given any representation $\rho:\G\to \SL(E)$, we denote by   $\r^{\dual}:\G\to \SL(E^{*})$  the \emph{dual (or contragradient) representation}; this is defined by the relation 
$$(\r^{\dual}(g)(w^{*}) )(v)= w^{*} (\rho(g^{-1})v)$$ for all $g\in\G,w^{*}\in E^{*}$ and $v\in E$.

\subsection{Anosov representations}
Anosov representations were introduced by Labourie for fundamental groups of negatively curved manifolds \cite{Labourie-IM} and generalized by Guichard-Wienhard to hyperbolic groups \cite{Guichard-Wienhard-IM}. Those representations yield generalizations of Teichm\"uller theory and convex cocompactness from rank one to higher rank. We will now recall the basic definitions recast in the framework of \cite{BPS}, which will be useful in the proof of Proposition \ref{prop.INTRO.projectionhyperconvex}

Given $A\in \SL(E)$ we denote by $|\lambda_1(A)|\geq \ldots \geq |\lambda_d(A)|$ the \emph{(generalized) eigenvalues} of $A$  counted with multiplicity and ordered non-increasingly in modulus. We will furthermore fix once and for all a scalar product on $E$ and denote by $\sigma_1(A)\geq \ldots \geq \sigma_d(A)$ the \emph{singular values} of the matrix $A$. That means that $\sigma_i(A)^2$ are the eigenvalues of the symmetric matrix $A^tA$.  

We fix a word metric on the Cayley graph of $\G$ for a fixed finite generating set of $\G$ and denote this by $|\cdot|_{\G}$. 

We will use the definition of Anosov representations from \cite{BPS}, which was shown to be equivalent to Labourie's and Guichard-Wienhard's original definition with different methods in \cite{KLP,BPS}:
\begin{defn}
An homomorphism $\rho:\G\to\SL(E)$ is $k$-Anosov if there exist positive constants $c,\mu$ such that, for all $\g\in\G$ 
$$\frac{\sigma_{k}}{\sigma_{k+1}}(\rho(\g))\geq ce^{\mu |\gamma|_\G}.$$
\end{defn}
The following properties follow easily from the definition: 
\begin{remark}\label{rem ansov rep properties} Let $\rho:\G\to\SL(E)$ be $k$-Anosov
\begin{enumerate}
\item The representation $\rho$ is faithful and has discrete image. Its orbit map is a quasi isometric embedding.
\item The representation $\rho$ is also $(d-k)-$Anosov: indeed $\sigma_{d-p}(\rho(\g^{-1}))=\sigma_p(\rho(\g))^{-1}$.
\end{enumerate}
\end{remark}
Furthermore it holds
\begin{prop}[\cite{BPS, KLP, Guichard-Wienhard-IM}]
The set of $k-$Anosov representation is open in $\Hom(\G,\SL(E))$.
\end{prop}

As already mentioned in the introduction, an important property of Anosov representations is that they admit continuous, transverse, dynamics preserving, equivariant boundary maps. This can be obtained as uniform limits of \emph{Cartan attractors}, as we now recall.

Every element $g\in\SL(E)$ admits a \emph{Cartan decomposition}, namely can be written uniquely as $g=k_g a_g l_g$ where $l_g,k_g\in\SO(E)$ and $a_g$ is diagonal with entries $\sigma_1(g),\ldots,\sigma_d(g)$. The $k$-th Cartan attractor is the subspace
$$U_k(g)=k_g\langle e_1,\ldots, e_k\rangle.$$
In other words $U_k(g)$ is a choice of the  $p$ longest axes of the ellipsoid $g\cdot B_1(0)\subset E$. Here $B_1(0)$ is the unit ball around the origin in $E$.
Observe that if $g$ has a gap of index $k$, i.e. $\sigma_k(g)>\sigma_{k+1}(g)$, then the $k$-th Cartan attractor doesn't depend on the choice of a Cartan decomposition.

Then the following holds:
\begin{prop}[{\cite[Proposition 4.9]{BPS}}]\label{p.bdry}
Let $\rho:\G\to \SL(E)$ be $k$-Anosov. Then for every geodesic ray $(\g_n)_{n\in\N}$ in $\G$ with endpoint $x\in\partial \G$ the limits
$$\xi^k(x):=\lim_{n\to \infty} U_k(\rho(\g_n))\quad \xi^{d-k}(x):=\lim_{n\to \infty} U_{d-k}(\rho(\g_n))$$
exist, do not depend on the ray and define continuous, $\rho$-equivariant, transverse maps $\xi^k:\partial\G\to \Gr_k(E)$,  $\xi^{d-k}:\partial\G\to \Gr_{d-k}(E)$.
\end{prop}
The uniformity of the limits in Proposition \ref{p.bdry} can be estimated explicitly  with respect to the distance on the Grassmannians induced by the fixed scalar product. To be more precise,  for $v,w\in E$ we let $\measuredangle (v,w)$ be the angle between the two vectors with respect to the chosen scalar product.  

The sine of the angle gives a distance, that we will denote by $d$, on the projective space $\P(E)$. More generally on every Grassmannian $\Gr_k(E)$ we set for $X,Y\in\Gr_k(E)$
$$d(X,Y):=\max_{v\in X^\times}\min_{w\in Y^\times}\sin\measuredangle(v,w)=\min_{v\in X^\times}\max_{w\in Y^\times}\sin\measuredangle(v,w),$$
where $X^\times=X\setminus\{0\}$, $Y^\times=Y\setminus\{0\}$.
This corresponds to the Hausdorff distance of $\P(X)$ and $\P(Y)$ regarded as subsets of $\P(E)$ with the aforementioned distance. 

Following Bochi-Potrie-Sambarino \cite{BPS} we further define the angle of two subspaces $X,Y<E$ as 
$$\measuredangle(X,Y)=\min_{v\in X^\times}\min_{w\in Y^\times}\measuredangle(v,w) $$
 Observe that in projective space $\sin\measuredangle (X,Y)=d(X,Y)$ while for general Grassmannians the inequality $\sin\measuredangle (X,Y)\leq d(X,Y)$ is, apart from very special cases, strict.

It then holds
\begin{prop}[cfr. {\cite[Lemma 4.7]{BPS}}]\label{p.bdrydist}
Let $\rho:\G\to \SL(E)$ be $k$-Anosov. Then there exist positive constants $C,\mu$ such that, for every geodesic ray  $(\g_n)_{n\in\N}$ starting at the identity with endpoint $x$ it holds
$$d(\xi^k(x),U_k(\rho(\g_n)))\leq Ce^{-\mu n}.$$
\end{prop}
Observe that, if $|\lambda_p(\gamma)|>|\lambda_{p+1}(\gamma)|$, then $U_p(\g^n)$ converges to the span of the first $p$ generalized eigenvalues, as a result one gets
\begin{prop}[{\cite{BPS}}]
Let $\r:\G\to \SL(E)$ be \emph{$k$-Anosov}.  Then $\xi^k$ and $\xi^{d-k}$ are \emph{dynamics preserving}, i.e. for every infinite order element $\g\in\G$ with attracting fixed point $\g^{+}\in\dg$ we have that $\xi^k(\g^{+})$ and $\xi^{d-k}(\g^{+})$ are attractive fixed points for the actions of $\r(\g^{+})$ on $\Gr_k(E)$ and $\Gr_{d-k}(E)$, respectively.
\end{prop}

\begin{notation}
Following the notation introduced in \cite{PSW}, we will often write $x_{\r}^k$ instead of $\xi^k(x)$ for the boundary map $\xi^k$ associated to  a $k$-Anosov representation $\rho$. If the representation is clear out of context, we will sometimes just write $x^k$.

Similarly we will write $g_{\r}$ instead of $\rho(g)$ for any $g\in \G$.
\end{notation}

We conclude the section by proving Proposition \ref{prop:continuity}, that provides the general setup useful to prove Proposition \ref{prop.INTRO.projectionhyperconvex}. For this we need another Lemma from \cite{BPS}:
\begin{lem}[{\cite[Lemma A.6]{BPS}}]\label{l.BPSA6}
Assume that $g\in\SL(E)$ has a gap of index $k$. Then for every $P\in\Gr_k(E)$ transverse to $U_{d-k}(g^{-1})$ it holds
$$d(gP,U_k(g))\leq \frac{\sigma_{k+1}}{\sigma_{k}}(g)\frac 1{\sin \measuredangle(P,U_{d-k}(g^{-1}))}.$$
\end{lem}
The last ingredient we will need in the proof of Proposition \ref{prop:continuity} is the concept of $\nu$-separated triples: we fix a distance $d$ on $\partial\G$ inducing the topology and say that an $n$-tuple $(x_1,\ldots,x_n)\in \partial \G$ is \emph{$\nu$-separated} if  $d(x_i,x_k)>\nu$ for all $i\neq k$. It follows from the properties of the convergence action of $\G$ on $\partial \G$ that there exists $\nu_0$ such that for every pairwise distinct triple $(a,b,c)\in\G$ there exists $g\in\G$ such that  $(ga,gb,gc)$ is $\nu_0$-separated. We will further assume, up to possibly shrinking $\nu_0$, that the endpoints of every biinfinite geodesic in $\Gamma$ through the origin are $\nu_0$-separated.

We now have all the tools we need to prove the only original result in the section. This is a generalization of the main argument in \cite[Proposition 6.7]{PSW}. We denote by $\dg^{(2)}$ the set of distinct pairs in the boundary of $\G$:
\begin{prop}\label{prop:continuity}
Let $\rho:\G\to\SL(E)$ be $k$-Anosov and $F:\dg^{(2)}\to \Gr_k(E)$ be continuous, $\rho$-equivariant. Assume that, for every pairwise distinct triple $x,y,z\in\partial\G$, 
$$ F(x,y)\tv z^{d-k}.$$ 
Then 
$$\lim_{y\to x} F(x,y)=x^k$$
\end{prop}
\begin{proof}
Observe that, as the set of $\nu_0$-separated triples is precompact, the assumption guarantees that there is $\epsilon$ such that, whenever $(x,y,z)$ is $\nu_0$ separated, it holds
$$\sin\measuredangle(F(x,y), z^{d-k})>\epsilon.$$

We choose a biinfinite geodesic $(\g_i)_{i\in \Z}$ through the origin with positive endpoint $x$, and denote by $z$ the negative endpoint of $(\g_i)_{i\in \Z}$.  Observe that for every $y$ there is $n=n_y$ such that $(\g_n^{-1}x,\g_n^{-1}y, \g_n^{-1}z)$ is $\nu_0$-separated (see \cite[Lemma 6.8]{PSW}). Furthermore $n_y$ goes to infinity as $y$ converges to $x$.

It follows from Proposition \ref{p.bdrydist} that, since $\rho$ is $(d-k)-$Anosov, and the ray $(\gamma_n^{-1}\g_{n-i})_{i\in\N}$ is a geodesic ray from the origin with endpoint $\gamma_n^{-1}z$, it holds 
$$d(U_{d-k}(\rho(\g_n^{-1})),\rho(\g_n^{-1})z^{d-k})\leq Ce^{-\mu n}.$$
Thus in particular, we can find $N$ such that, for every $y$ such that  $n_y>N$, we have
$$\begin{array}{l}
\sin\measuredangle(F(\g_n^{-1}x,\g_n^{-1}y), U_{d-k}(\rho(\g_n^{-1})))\geq\\
\sin\measuredangle(F(\g_n^{-1}x,\g_n^{-1}y), \rho(\g_n^{-1})z^{d-k})-d(U_{d-k}(\rho(\g_n^{-1})),\rho(\g_n^{-1})z^{d-k})\geq \frac\epsilon 2.
\end{array}$$
Lemma \ref{l.BPSA6} ensures
$$d(F(x,y),U_k(\rho(\g_n)))\leq \frac{\sigma_{k+1}}{\sigma_{k}}(\rho(\g_n))\frac 2{\epsilon}\leq \frac{2e^{-\mu n}}{\epsilon c}. $$
The result then follows from Proposition \ref{p.bdrydist} using the triangle inequality.
\end{proof}

\section{Cross ratios}\label{s.cr}

An important tool will be cross ratios, which we introduce here. We will need two different notions of cross ratio and their relation.

\subsection{Projective cross ratios}
Probably the most classical notion of cross ratio is the \emph{projective cross ratio} $\pc$ on $\R\Pp^1$. This cross ratio can be defined by
\[\pc(x_1,x_2,x_3,x_4):=\frac{\tilde{x}_1\wedge \tilde{x}_3}{\tilde{x}_1\wedge \tilde{x}_2}\frac{\tilde{x}_4\wedge \tilde{x}_2}{\tilde{x}_4\wedge \tilde{x}_3}\] 
if no three of the four $x_i\in \R\Pp^1$ are equal and $\tilde{x}_i\in\R^2\backslash\{0\}$ are non-trivial lifts. We also need to choose an identification $\wedge^2\R^2\simeq \R$, but the definition is independent of all choices made.

\begin{lem}\label{lem.symmetry of projective cro}
Let $x_1,\ldots,x_5\in \R\Pp^1$. Then whenever all quantities are defined we have
\begin{enumerate}
\item $\pc (x_1,x_2,x_3,x_4)^{-1}=\pc (x_4,x_2,x_3,x_1)=\pc (x_1,x_3,x_2,x_4)$
\item $\pc (x_1,x_2,x_3,x_4)\cdot\pc (x_4,x_2,x_3,x_5)=\pc (x_1,x_2,x_3,x_5)$
\item $\pc (x_1,x_2,x_3,x_4)\cdot\pc (x_1,x_3,x_5,x_4)=\pc (x_1,x_2,x_5,x_4)$
\item $\pc (x_1,x_2,x_3,x_4)=0\Longleftrightarrow x_1 = x_3$ or $x_4= x_2$
\item $\pc (x_1,x_2,x_3,x_4)=1\Longleftrightarrow x_1 = x_4$ or $x_2= x_3$
\item $\pc (x_1,x_2,x_3,x_4)=\infty\Longleftrightarrow x_1 = x_2$ or $x_4= x_3$
\item $\pc (x_1,x_2,x_3,x_4)=\pc (g x_1,g x_2,g x_3,g x_4)\quad  \forall g\in\SL(\R^2)$
\item $\pc (x_1,x_2,x_3,x_4)=1-\pc (x_1,x_2,x_4,x_3)$.
\item If $x_1,\ldots,x_5$ are  cyclically ordered, then\label{e.mon}
\begin{align*}
\pc (x_1,x_2,x_3,x_5)&<\pc (x_1,x_2,x_4,x_5)\\
\pc (x_1,x_3,x_4,x_5)&<\pc (x_2,x_3,x_4,x_5).
\end{align*} 
\item $\pc (x_1,x_2,x_3,x_4)>1$ if and only if $(x_1,x_2,x_3,x_4)$ is cyclically ordered.
\end{enumerate}
\end{lem}

All properties are straight forward to check (and well known). 

We will later use the following consequence of Lemma \ref{lem.symmetry of projective cro} \eqref{e.mon}:
\begin{lem}\label{lem.non-zero derivative of projective cross ratio}
Let $c:I\subset \R \to \R\Pp^1$ be $C^1$ at $i_0\in I$. If $c(i_0),x_1,x_2\in \R\Pp^1$ are pairwise distinct and  $dc_{|i_0}\neq 0$, then
$$\left.\frac{d}{dt}\right|_{t=i_0} \pc(c(i_0),x_1,x_2,c(t))\neq 0.$$
\end{lem}
The projective cross ratio can be used to define a cross ratio on pencils of vector subspaces:
\begin{defn}\label{def.quotient cross ratio}
Let $V^{k-1}\in \Gr_{k-1}(E),$ $V^{k+1}\in \Gr_{k+1}(E)$ and $W_i^k\in \Gr_k(E)$ for $i=1,\ldots,4$ such that $V^{k-1}<W_i^k<V^{k+1}$.   We set
$$ \pc_{V^{k-1}}\left(W_1^k,W_2^k,W_3^k,W_4^k\right):=\pc\left([W_1^k],[W_2^k],[W_3^k],[W_4^k]\right),$$
where $[W_i^k]\in \Pp(\quotient{V^{k+1}}{ V^{k-1}})\simeq \R\Pp^1$ is the projection.
\end{defn}

We will also allow entries of the form $w\in \Pp(V^{k+1}\backslash V^{k-1})$  in the left hand side, as this defines via $w\oplus V^{k-1}$ a $k-$vector space satisfying the assumption of Definition \ref{def.quotient cross ratio}.

This cross ratio is useful to determine the root gap: 

\begin{prop}[{cfr. \cite[Lem. 2.9]{LZ}}]\label{p.gc1}
Let $\r$ be $\{k-1,k,k+1\}-$Anosov. 
Then for every non-trivial $h\in \G$
$$\frac{\lambda_k(h_{\r})}{\lambda_{k+1}(h_{\r})}=
\pc_{h_-^{d-k-1}}\left(h_-^{d-k},x^k\cap h_-^{d-k+1}, hx^k\cap h_-^{d-k+1}, h_+^k\cap h_-^{d-k+1}\right)$$
for any $x\in \dg\backslash \{h_{\pm}\}$.
\end{prop}

\begin{center}
\begin{figure}[h]
\begin{tikzpicture}
\draw (0,0) circle [radius =1];
\draw (0,-1) to (0,1.5)[thick, green];
\draw (-1,0) to (1,0);
\draw (0,-1) to (0,1);
\node at (.5,0)  {$>$};
\node at (0,.5) [rotate=90] {$>$};
\filldraw (1,0) circle [radius=1pt];
\node at (1,0) [right] {$x^k\cap h_-^{d-k+1}$};
\draw (0,-1) to (1.5,0.5)[thick, green];
\filldraw (0,1) circle [radius=1pt];
\node at (0,1) [above] {$h^k_+\cap h_-^{d-k+1}$};
\filldraw (0,-1) circle [radius=1pt];
\node at (0,-1) [below] {$h_-^{d-k-1}$};
\node at (2,-1) [below] {$h_-^{d-k}$};
\draw (-1,-1) to (1.5,-1)[thick, green];
\node at (.8,.8)[right]{$hx^k\cap h_-^{d-k+1}$};
\draw (0,-1) to (1,1.3)[thick, green];
\draw (0,-1) to (0,1)[thick, green];
\filldraw (.7,.7) circle [radius=1pt];
\end{tikzpicture}
\caption{Schematic picture of the statement of Proposition \ref{p.gc1}, all thick green lines are to be understood as subspaces of $h_-^{d-k+1}$; their cross ratio is the $k$-th eigenvalue gap}
\end{figure}
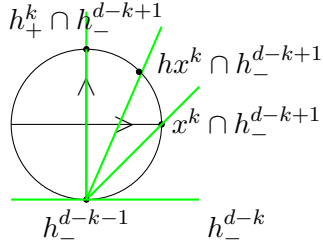
\end{center}

\begin{proof}
Pick a generalized eigenbasis $(e_1,\ldots,e_d)$ of $h_{\r}$ such that $e_i$ corresponds to $\lambda_i(h_{\r})$ with the $\lambda_i$ ordered decreasingly in modulus as usual. Then $\quotient{h_{-}^{d-k+1}}{ h^{d-k-1}_{-}}\simeq \langle e_{k},e_{k+1}\rangle$. Moreover by the Anosov condition $e_{k},e_{k+1}$ are eigenvectors (and not only generalized eigenvectors). Thus if $[\cdot]$ denotes the projection to $\Pp(\quotient{h_{-}^{d-k+1}}{ h^{d-k-1}_{-}})$, we get in the basis $[e_{k}],[e_{k+1}]$ that
$$
\begin{array}{lcl}
[h_+^k\cap h_-^{d-k+1}]=[e_k],&& [h_-^{d-k}]=[e_{k+1}], \\ 
{[x^k\cap h_-^{d-k+1}]=\left[\bpm a\\ b\epm\right],} &&{[hx^k\cap h_-^{d-k+1}]=\left[\bpm \lambda_k(h_{\r}) a\\ \lambda_{k+1}(h_{\r})b\epm\right]}
\end{array}
$$
for some $a,b\in\R\backslash\{0\}$. The claim follows through a short calculation.
\end{proof}

\subsection{Grassmannian cross ratio}

The projective cross ratio has a generalization to Grassmannians, which we now describe. Set
\begin{align*}
\calA_k\: :=\{ (V_1,W_2,W_3,V_4) | &V_1,V_4\in \Gr_k(E), W_2,W_3\in \Gr_{d-k}(E) \text{ and } V_j\tv W_i\\
 &\text{ for } (j,i)=(1,2),(4,3) \text{ or } (j,i)=(1,3), (4,2)\}
\end{align*}

\begin{defn}\label{def.k cross ratio}
Let $(V_1,W_2,W_3,V_4)\in\calA_k$. Then the (generalized) cross ratio $\gc_k:\calA_k\to \R\cup\{\infty\}$ 
is defined by
\begin{align*}
\gc_k (V_1,W_2,W_3,V_4):=\frac{V_1\wedge W_3}{V_1\wedge W_2} \frac{V_4\wedge W_2}{V_4\wedge W_3},
\end{align*}
where $V_i\wedge W_j$ denotes the element $v_1\wedge \ldots\wedge v_k \wedge w_1\wedge \ldots \wedge w_{d-k}\in \wedge^d E \simeq \R$ for fixed bases $(v_1,\ldots, v_k),(w_1,\ldots, w_{d-k})$  of $V_i$ and $W_j$, respectively, and a fixed identification $\wedge^d E \simeq \R$. Note that the value of $\gc_k$ is independent of all choices made.
Here we use the convention $\frac{a}{0} :=\infty$ for any non-zero $a\in\R$.
\end{defn}

\begin{remark}
The modulus of this (generalized) cross ratio is a special case of the, in general vector valued, cross ratios on flag manifolds $G\slash P$ constructed in \cite{Beyrer} - see Example 2.11 therein.
\end{remark}

\begin{lem}\label{lem property of grassmannian cro}
Let $V_1,V_4,V_5\in\Gr_k(E)$ and $W_2,W_3,W_5\in\Gr_{d-k}(E)$. Then whenever all quantities are defined we have
\begin{enumerate}
\item $\gc_k (V_1,W_2,W_3,V_4)^{-1}=\gc_k (V_4,W_2,W_3,V_1)=\gc_k (V_1,W_3,W_2,V_4)$
\item $\gc_k (V_1,W_2,W_3,V_4)\cdot\gc_k (V_4,W_2,W_3,V_5)=\gc_k (V_1,W_2,W_3,V_5)$
\item $\gc_k (V_1,W_2,W_3,V_4)\cdot\gc_k (V_1,W_3,W_5,V_4)=\gc_k (V_1,W_2,W_5,V_4)$
\item $\gc_k (V_1,W_2,W_3,V_4)= 0 \Longleftrightarrow V_1 \ntv W_3$ or $V_4\ntv W_2$
\item $\gc_k (V_1,W_2,W_3,V_4)=\infty\Longleftrightarrow V_1 \ntv W_2$ or $V_4\ntv W_3$
\item $\gc_k (V_1,W_2,W_3,V_4)=\gc_k (g V_1,g W_2,g W_3,g V_4)\quad  \forall g\in\SL(E)$
\end{enumerate}
\end{lem}

\begin{notation} Given a $k-$Anosov representation $\r$ and pairwise distinct $x,y,z,w\in \dg$. Then $(x^k,y^{d-k},z^{d-k},w^k)\in\calA_k$. In this case we write
\[\gc_k(x,y,z,w):=\gc_k(x^k,y^{d-k},z^{d-k},w^k).\]
\end{notation}

The $k-$cross ratio can give information on the eigenvalues of specific elements:
We say that an element $A\in \SL(E)$ \emph{has an eigenvalue gap of index $k$}  if  $|\lambda_k(A)|>|\lambda_{k+1}(A)|$. In this case, we denote by $A^+_k\in\Gr_k(E)$ the span of the first $k$ generalized eigenspaces. Furthermore if $A$ also has an eigenvalue gap of index $d-k$  we denote by $A^-_k:=(A^{-1})^+_k\in\Gr_k(E)$. 

Then following is easy to check:
\begin{lem}
If $A$ has eigenvalue gaps of indices $k,d-k$ then for every $V\in \Gr_{k}(E)$ transverse to $A_{d-k}^\pm$, and $W\in \Gr_{d-k}(E)$ transverse to $A_{k}^\pm$ it holds
$$\gc_k (A^-_k,W,AW,A^+_{k})=\gc_k (V,A^-_{d-k},A^+_{d-k},AV)=\frac{\lambda_1(A) \ldots \lambda_k(A) }{\lambda_d(A) \ldots \lambda_{d-k+1}(A)}$$ 
\end{lem}

Note that, since the boundary map of a $k-$Anosov representation is dynamics preserving, we have that $g_{\r}$ has eigenvalue gaps of indices $k,d-k$ for every non-trivial $g\in\G$. Furthermore $\xi^k(g_{+})=(g_{\r})^+_k, \xi^k(g_{-})=(g_{\r})^-_k$  for $g_{\pm}$ attractive and repulsive fixed points of $g$ respectively. This yields: 

\begin{cor}\label{cor.periods of k cross ratio}
Let $\r$ be $k-$Anosov. Then for every non-trivial $g\in\G$ and every $x\in\dg\backslash\{g_{\pm}\}$ we have
$$\gc_k(g_{-},x,gx,g_{+})=\gc_k(x,g_{-},g_{+},gx)=\frac{\lambda_1(g_{\r}) \ldots \lambda_k(g_{\r}) }{\lambda_d(g_{\r}) \ldots \lambda_{d-k+1}(g_{\r})}>1.$$
\end{cor}

\begin{proof}
It remains to show that $\gc_k(g_{-},x,gx,g_{+})>1$. Observe that $k-$Anosov yields that $$\left|\frac{\lambda_1(g_{\r}) \ldots \lambda_k(g_{\r}) }{\lambda_d(g_{\r}) \ldots \lambda_{d-k+1}(g_{\r})}\right|>1.$$
For $y$ in the same connected component of $\dg\backslash\{g_{\pm}\}$ as $x$ we define the map $y\mapsto gc_k(g_{-},x,y,g_{+})$, which is continuous and never zero. As $gc_k(g_{-},x,x,g_{+})=1$, the image of this map is in $\R_{>0}$. In particular 
$$\gc_k(g_{-},x,gx,g_{+})>0,$$
which yields the claim.
\end{proof}

\subsection{Relations of the cross ratios}

We now investigate relations between the various cross ratios introduced so far. For this it will be useful to remember the projection $[\cdot]_X$ introduced in Notation \ref{n.1}.

\begin{prop}\label{prop.k-cross ratio and projection}
Let $P^k\neq Q^k\in \Gr_k(E)$. Denote by $X^{k+l}:=\<P^k,Q^k\>$ their span, by $X^{k-l}:=P^k\cap Q^k$ their intersection and by $X:=\quotient{X^{k+l}}{X^{k-l}}$ the quotient. For each $(d-k)$-dimensional subspaces $S^{d-k},T^{d-k}$ transverse to $X^{k-l}$ it holds 
$$\gc_k\left(P^k,\,S^{d-k},T^{d-k},Q^k\right)=\gc_l([P^k]_X, [S^{d-k}\cap X^{k+l}]_X,[T^{d-k}\cap X^{k+l}]_X,[Q^k]_X),$$
whenever one of the sides is defined.
\end{prop}
\begin{proof}
This is a direct computation: if we pick a basis $(e_1,\ldots,e_d)$ such that 
$$X^j=\langle e_1,\ldots, e_{j}\rangle,\quad j=k-l,k+l$$ 
and choose a basis of $S^{d-k}$ (resp. $T^{d-k}$) whose first $l$ vectors belong to $X^{k+l}$, in order to compute the cross ratio on the left hand side we have to compute the determinant of four block upper triangular matrices, whose first blocks are always the identity, the third blocks cancel between the numerator and denominator, and the remaining blocks (of size $2l\times 2l$) gives the desired cross ratio on the right.  
\end{proof}

Since by definition $\gc_1=\pc$ on $\P(\R^2)$, we immediately get

\begin{cor}\label{cor.projective and k-cross ratio}
Let $P^k,Q^k\in \Gr_k(E)$ be such that $\dim P^k\cap Q^k=k-1$. Then, using the notation of the Proposition \ref{prop.k-cross ratio and projection}, it holds 
$$\gc_k\left(P^k,\,S^{d-k},T^{d-k},Q^k\right)=\pc([P^k]_X, [S^{d-k}\cap X^{k+1}]_X,[T^{d-k}\cap X^{k+1}]_X,[Q^k]_X),$$
whenever one of the sides is defined. 
\end{cor}

In the special case of points and hyperplanes in $\Pp(E)$  the above connection works for all transverse points

\begin{cor}\label{cor.1cross ratio as projective one}
Let $p^1,q^1\in\Pp(E)$ be transverse to $V^{d-1},W^{d-1}\in \Gr_{d-1}(E)$. Then
$$\gc_1(p^1,V^{d-1},W^{d-1},q^1)=\pc_{V^{d-1}\cap W^{d-1}}(p^1,V^{d-1},W^{d-1},q^1).$$
\end{cor}

\section{Partial hyperconvexity}\label{s.hyp}

\subsection{Property $H_k$}
The following transversality property was introduced by Labourie \cite[Section 7.1.4]{Labourie-IM} in the context of Hitchin representations, and generalized to other groups in \cite[Section 8.2]{PSW}.
\begin{defn}[{\cite[Section 7.1.4]{Labourie-IM}}]
A representation $\rho:\Gamma\to\SL(E)$ satisfies \emph{property $H_k$} if it is $\{k-1,k,k+1\}-$Anosov and the following sum is direct
\begin{align}\label{eq property H}
x^k+ \left(y^k\cap z^{d-k+1}\right)+ z^{d-k-1}.
\end{align} 
\end{defn}
\begin{figure}[h]\label{fig:H}
\begin{tikzpicture}
\draw (0,0) circle [radius=1cm];
\filldraw (1,0) circle [radius= 1pt] node[right] {$z^{d-k-1}$};
\draw (1,-1.5) to (1, 2);
\node at (1, 2) [above] {$z^{d-k+1}$};
\node at (-.2,.97) [above] {$x^{k}$};
\node at (-.2,-.97) [below] {$y^{k}$};
\draw (-1,.8) to (1.3,1.25);
\draw (-1,-.8) to (1.3,-1.25);
\filldraw (1, 1.17) circle [radius=1pt] node[above right]{$x^k\cap z^{d-k+1}$};
\filldraw (1, -1.17) circle [radius=1pt] node[below right]{$y^k\cap z^{d-k+1}$};
\end{tikzpicture}
\caption{Property $H_k$}
\end{figure}
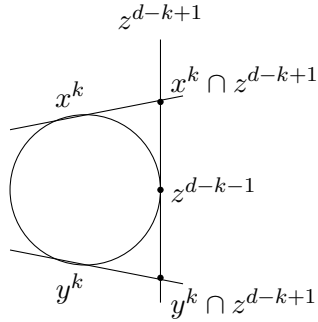
Property $H_k$ has the following equivalent characterization, which follows from the transversality properties guaranteed by $(k-1)-$Anosov:
\begin{align}\label{eq.char of Hk}
 \left(x^k\cap z^{d-k+1}\right)\oplus\left( y^k\cap z^{d-k+1}\right)\oplus z^{d-k-1}  = z^{d-k+1}.
\end{align}

The following result from \cite{PSW} will be important for us  (see also \cite{ZZ}):
\begin{prop}[{\cite[Proposition 8.11]{PSW}}]\label{pro 1-hyperconvex for property Hk}Let $\rho:\Gamma\to\SL(E)$ satisfy property $H_k$. Then the boundary curve $\xi^k$ has $C^1-$image and the tangent space is given by
\begin{align*}
T_{x^k}\xi^k(\dg)=\{ \phi\in \Hom(x^k,y^{d-k})| x^{k-1}\subseteq \ker \phi, \;{\rm Im}\: \phi\subseteq  x^{k+1}\cap y^{d-k}  \}
\end{align*}
for any $y\neq x \in\dg$.
\end{prop}
The tangent space $T_{x^k}\Gr_k(E)$ has a natural identification with $\Hom(x^k,y^{d-k})$ for every $y^{d-k}\in \Gr_{d-k}(E)$ transverse to $x^k$, and the above proposition is independent on the choice of $y\neq x \in\dg$. With a slight abuse of notation we will not distinguish between $T_{x^k}\Gr_k(E)$ and $\Hom(x^k,y^{d-k})$.

Proposition \ref{pro 1-hyperconvex for property Hk} can be rephrased as saying that, whenever $\rho$ satisfies property $H_k$, at the point $x^k$, the differentiable curve $\xi^k$ is tangent to the \emph{pencil} 
$$\R\P^1=\P\left(\quotient{x^{k+1}}{x^{k-1}}\right):=\left\{Z\in\Gr_k(E)| \; x^{k-1}\subset Z\subset x^{k+1}\right\}.$$
\begin{defn}\label{d.linearapp}
A \emph{linear approximation} of $\xi^k$ at $x$ is a smooth monotone map 
$x_t:(-\epsilon, \epsilon)\to \P(\quotient{x^{k+1}}{x^{k-1}})$ with $x_0=x^k$ and non vanishing derivative at $0$.
\end{defn}

As the Anosov property is open in $\Hom(\G,\SL(E))$ and property $H_k$ is a transversality condition on the set of triples of points in the boundary, one immediately gets (cfr. \cite[Proposition 8.2]{Labourie-IM}):
\begin{prop}\label{prop Hk open property}
The set of representations satisfying property $H_k$ is open in $\Hom(\G,\SL(E))$.  
\end{prop}

Recall that we denote by $\r^{\dual}$ the contragradient representation. The following proposition follows from a straight forward computation:

\begin{prop}\label{prop Hk for dual representation}
A representation $\r:\G\to \SL(E)$ satisfies properties $H_k$ if and only if $\r^{\dual}$ satisfies property $H_{d-k}$.
\end{prop}

Thus if the Zariski closure of $\r(\G)$ is contained in $\Sp(2n,\R)$ or $\SO(p,q)$, then $\r$ satisfies property $H_k$ if and only if it satisfies property $H_{d-k}$.

Property $H_k$ can be restated as an injectivity property of the natural projections of the boundary on the projective lines $\Pp(\quotient{x^{d-k+1}}{ x^{d-k-1}})$:
\begin{defn}\label{def.Px}
Let $\r$ be a $\{k-1,k,k+1\}-$Anosov representation, fix  a point $x\in\dg$ and set $X:=\quotient{x^{d-k+1}}{ x^{d-k-1}}$. 

The \emph{$k$--th $\R\P^1$-projection} based at $x$ is the map 
$$P_x:\dg\to \Pp(\quotient{x^{d-k+1}}{ x^{d-k-1}})$$ given by 
$$P_x(y)=\begin{cases}
\; [y^k\cap x^{d-k+1}]_X &y\neq x,\\
 \;[x^{d-k}]_X &y=x
\end{cases} $$
%
%
%
\end{defn}

\begin{prop}\label{prop.P projec injective}
A $\{k-1,k,k+1\}-$Anosov representation satisfies property $H_k$ if and only if $P_x$ is injective for all $x\in \dg$.
\end{prop}

\begin{proof}
This is essentially by definition. One only has to note that $k-$Anosov guarantees already that $P_x(y)\neq P_x(x)$ for $y\neq x$.
\end{proof}

A representation $\r:\G\to \SL(E)$ is called \emph{strongly irreducible} if the restriction of $\r$ to any finite index subgroup $\G'<\G$ is an irreducible representation.

\begin{prop}
The subset of representations satisfying property $H_k$ is a union of connected components of strongly irreducible $\{k-1,k,k+1\}-$Anosov representations.
\end{prop}

\begin{proof}
Since $P_x:\dg\to\Pp(\quotient{x^{d-k+1}}{ x^{d-k-1}})$ is injective, the same argument as in \cite[Proposition 9.3]{PSW} implies that this is a closed condition within the set of representations for which $P_x$ is not locally constant.  The fact that $\rho$ is strongly irreducible guarantees that $P_x$ is not locally constant \cite[Lem 10.2]{Labourie-IM}. The claim follows.
\end{proof}

For property $H_k$ we get additionally the following properties for $P_x$.

\begin{prop}\label{prop.properties of proj Px}
Let $\r$ satisfy property $H_k$. Then
\begin{enumerate}
\item $P_x $ is continuous;
\item For every cyclically ordered $n$-tuple $(y_1,\ldots,y_n)\in \dg^n$ the $n$-tuple $(P_x(y_1),\ldots, P_x(y_n))$ is cyclically ordered in $\Pp(\quotient{x^{d-k+1}}{ x^{d-k-1}})$.
\end{enumerate}
\end{prop}

\begin{proof}
$(1)$ The continuity at $\dg\backslash\{x\}$ follows directly from the continuity of $y\mapsto y^k$. 
 Property $H_k$ guarantees that 
$$F(x,y):=\left((y^k\cap x^{d-k+1})\oplus x^{d-k-1}\right) \tv z^k $$
for every pairwise distinct triple $x,y,z$. Since $F$ is $\rho$-equivariant, Proposition \ref{prop:continuity} implies that
$$F(x,y)=\left((y^k\cap x^{d-k+1})\oplus x^{d-k-1}\right) \to x^{d-k}$$
for $y\to x$.
If we project this to $\Pp(\quotient{x^{d-k+1}}{ x^{d-k-1}})$ we get $P_{x}(y)\to P_{x}(x)$ for $y\to x$ as desired.

$(2)$ Since $P_x$ is an injective continuous map between topological circles, it is a homeomorphism. Therefore it preserves the cyclic order on all of $\dg$.
\end{proof}

\begin{remark}\label{rem.Hk and projective cross ratio}
Note that the function on 4-tuples of points in $\dg$ given by $(x,y,z,t)\mapsto \pc(P_x(u),P_x(y),P_x(z),P_x(w))$ has  the symmetries of the projective cross ratio as in Lemma \ref{lem.symmetry of projective cro} ($(1)-(3), (8)$). By Proposition \ref{prop.properties of proj Px} it also inherits $(4)-(6), (9),(10)$ of that Lemma.
\end{remark}

\begin{cor}\label{cor.roots are positive for Hk}
If $\r$ satisfies property $H_k$, then for every $h\in\G\backslash\{e\}$ we have $\lambda_k(h_{\r})\slash \lambda_{k+1}(h_{\r})>0$.
\end{cor}

\begin{proof}
For any $y\in\dg\backslash\{h_{\pm}\}$, we have that $h_+,hy,y,h_-$ are in that cyclic order on $\dg$. Thus Proposition \ref{prop.properties of proj Px} implies that
$$ P_{h_-}(h_+), P_{h_-}(hy), P_{h_-}(y), P_{h_-}(h_-)$$ are in that cyclic order on $\Pp(\quotient{h_-^{d-k+1}}{ h_-^{d-k-1}})\simeq \R\Pp^1$. Thus by Proposition \ref{p.gc1} together with  Lemma \ref{lem.symmetry of projective cro} $(10)$ it follows that 
$$
\frac{\lambda_k(h_{\r})}{\lambda_{k+1}(h_{\r})}=
\pc\left(P_{h_-}(h_+), P_{h_-}(hy), P_{h_-}(y), P_{h_-}(h_-)\right)>0. \qedhere
$$
\end{proof}

\subsection{Property $C_k$} We will need also to consider representations that satisfy a bit more transversality of the boundary maps than property $H_k$. We introduce here this new notion.

\begin{defn}
A representation $\r:\G\to \SL(E)$ satisfies \emph{property $C_k$} if it is $\{k-1,k,k+1,k+2\}-$Anosov and for pairwise distinct  $x,y,z\in \dg$ the sum 
$$x^{d-k-2} + (x^{d-k+1}\cap y^{k})+ z^{k+1}$$ 
is direct.
\end{defn}

Note that by transversality of the boundary maps property $C_k$ is equivalent to 
$$(y^k\cap x^{d-k+1})\oplus (z^{k+1}\cap x^{d-k+1})\oplus x^{d-k-2} =x^{d-k+1}$$ 
for all pairwise distinct $x,y,z\in \dg$. 
The special case of $k=1$, i.e. $x^{d-3} \oplus y^{1} \oplus z^{2}=E$, is referred to as $(1,2,3)-$hyperconvex in \cite{PSW}.

Since property $C_k$ is a transversality property on triples of points, the following is proved in the same way as Proposition \ref{prop Hk open property}.
\begin{prop}
Property $C_k$ is an open condition in $\Hom(\G,\SL(E))$.
\end{prop}

Again a straight forward computation yields:

\begin{prop}
A representation $\r$ satisfies property $C_k$ if and only if $\r^{\dual}$ satisfies $C_{d-k}$.
\end{prop}

Property $C_k$ has consequences on the following naturally defined projections on the projective planes $\Pp(\quotient{x^{d-k+1}}{ x^{d-k-2}})$:
\begin{defn}\label{def.RP2 projection} 
Let $\r:\G\to \SL(E)$ be $\{k-1,k,k+1,k+2\}-$Anosov and fix a point $x\in\dg$. 
We denote by $X=X_\r$ the three dimensional vector space $\quotient{x^{d-k+1}}{ x^{d-k-2}}$, and denote by $\calF(X)$ the space of complete flags of $X$. 

The \emph{$k$-th $\R\P^2$ projection} associated to $x$ is the flag map
$\pi_x:\dg\to\calF(X)$ given by
$$
\pi_x^{(1)}(y)=
\begin{cases}
\;[y^k\cap x^{d-k+1}]_X &y\neq x\\
\;[x^{d-k-1}]_X &y= x\\
\end{cases}
$$
$$
\pi_x^{(2)}(y)=
\begin{cases}
\;[y^{k+1}\cap x^{d-k+1}]_X &y\neq x\\
\;[x^{d-k}]_X &y= x.\\
\end{cases}
$$
\end{defn}

Recall from Labourie \cite{Labourie-IM} the following
\begin{defn}
A continuous curve $\xi:\dg\to\R\Pp^2$ is  \emph{hyperconvex} if $\xi(x)\oplus \xi(y)\oplus \xi(z)=\R^3$ for all pairwise distinct $x,y,z\in \dg$. 
\end{defn}
The combination of $H_k$ and $C_k$ ensures that all $k$-th $\R\P^2$-projections are  hyperconvex :

\begin{prop}\label{prop.hyperconvex proj to RP2}
If $\r:\G\to\SL(E)$  satisfies properties $H_{k}$ and $C_k$, then for each $x\in\dg$, $\pi_x^{(1)}:\dg\to\P(X)$ defines a continuous hyperconvex curve  that admits a $C^1$ parametrization with tangents  $\pi_x^{(2)}$. 
\end{prop}

\begin{proof}
We first show the transversality:
If the triple is of the form $x,y,z$, the sum $\pi_x^{(1)}(x)+\pi_x^{(1)}(y)+\pi_x^{(1)}(z)$ is direct if and only if the sum
$$  (y^k\cap x^{d-k+1})+ (z^k\cap x^{d-k+1})+ x^{d-k-1} $$
is direct, which holds as $\r$ satisfies property $H_k$.

Let now $x,y,z,w\in\dg$  be pairwise distinct. We can assume without loss of generality that the points are in that cyclic order.  We set $X':= \quotient{X}{\pi_x^{(1)}(w)}$. The claims follows if we show that the projection 
\begin{align*}
\pi_{x,w}:\dg\backslash \{x,w\}\to & \Pp (X')\simeq \R\P^1\\
u\mapsto& [\pi_x^{(1)}(u)]_{X'}
\end{align*}
restricted to any connected component of $\dg\backslash \{x,w\}$ is injective. Note that property $H_k$ guarantees that the projection is well defined. Since property $C_k$ implies that $\quotient{\pi_x^{(2)}(w)}{\pi_x^{(1)}(w)}$ is not in the image, the map $\pi_{x,w}$ is a continuous map from a topological interval to a topological interval. Hence it is enough to check local injectivity. 

By property $H_{k}$ it follows that $u\mapsto (u^k\cap x^{d-k+1})$ has $C^1-$image with tangent given by $u^{k+1}\cap x^{d-k+1}$ (Proposition \ref{pro 1-hyperconvex for property Hk}). Thus, by construction, outside $x$, the map $\pi_x^{(1)}$ has $C^1$ image with derivative $\pi_x^{(2)}$. Since,  by property $C_k$, we have 
$$[\pi_x^{(2)}(u)]_{X'}=\P(X'),$$
or equivalently $\pi_x^{(2)}(u)$ never belongs to the kernel of the projection to $X'$, it follows that the tangent space of the image of $\pi_{x,w}$ is nowhere degenerate. This implies local injectivity and thus yields transversality.

The continuity of $\pi_x$ at a point $y\in \dg\backslash\{x\}$ is clear. We are left to show that those maps extend continuously at $x$. This will be guaranteed by Proposition \ref{prop:continuity}: we define the maps 
\begin{align*}
F^{d-k+1}(x,y):=\left(y^{k}\cap x^{d-k+1}\right) \oplus x^{d-k-2}\in\Gr_{d-k-1}(V)\\
F^{d-k}(x,y):=\left(y^{k+1}\cap x^{d-k+1}\right) \oplus x^{d-k-2}\in\Gr_{d-k}(V).
\end{align*}
Property $C_{k}$ guarantees that  for every pairwise distinct triple $(x,y,z)$ it holds
\begin{align*}
z^{k+1}\tv \left(y^{k}\cap x^{d-k+1}\right) \oplus x^{d-k-2}\\
z^{k}\tv \left(y^{k+1}\cap x^{d-k+1}\right) \oplus x^{d-k-2}.
\end{align*}
Thus Proposition \ref{prop:continuity} applies to $F^{d-k+1}, F^{d-k}$ and this yields the continuity of $\pi_x$ at $x$.

Since the  map $\pi_x^{(1)}$ is the composition of the boundary map $\xi^k$, which has $C^1$ image, with a smooth projection, and we verified that the derivative of a $C^1$ parametrization of $\xi^k$ never belongs to the kernel of the projection, and has image $\pi_x^{(2)}$, it follows that $\pi_x^{(1)}$ admits a $C^1$ parametrization with tangents $\pi_x^{(2)}$ outside $x$. The analogous statement at $x$ follows by continuity.
\end{proof}

The projections $\pi_x$ are also useful to show that property $C_k$ is closed within the space of Anosov representations:
\begin{prop}
Property $C_k$ is closed among strongly irreducible $\{k-1,k,k+1,k+2\}-$Anosov representations satisfying property $H_{k}$.
\end{prop}

\begin{proof}
Let $\{\r_n\}$ be a sequence of $\{k-1,k,k+1,k+2\}-$Anosov representations satisfying property $H_{k}$ and $C_k$ converging to $\r_0$, a strongly irreducible $\{k-1,k,k+1,k+2\}-$Anosov representation satisfying property $H_{k}$.
Denote, as above, $X_{\rho_0}:=\quotient{x_{\r_0}^{d-k+1}}{x_{\r_0}^{d-k-2}}$. 

Since $\r_0$ is strongly irreducible, it follows by \cite[Lem 10.2]{Labourie-IM} that, for fixed $z\neq x$, the set of $u\in\dg$ for which 
\begin{align}\label{eq.proof Ck closed0}
\pi_{x_{\rho_0}}^{(1)}(u)\ntv \pi_{x_{\rho_0}}^{(2)}(z)\quad\text{or}\quad \pi_{x_{\rho_0}}^{(2)}(u)\ntv \pi_{x_{\rho_0}}^{(1)}(z)
\end{align}
doesn't contain open intervals: indeed the two conditions are equivalent respectively to 
\begin{align*}
u_{\r_0}^k\ntv \left(x_{\r_0}^{d-k-2}\oplus (z_{\r_0}^{k+1}\cap x_{\r_0}^{d-k+1})\right), \quad u_{\r_0}^{k+1}\ntv \left(x_{\r_0}^{d-k-2}\oplus (z_{\r_0}^{k}\cap x_{\r_0}^{d-k+1})\right),
\end{align*}
conditions that, by \cite[Lem 10.2]{Labourie-IM} violate strong irreducibility if happening on an open intervals.

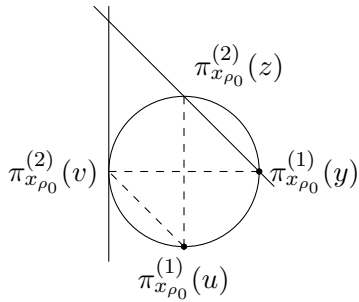
\begin{figure}[h]
\begin{tikzpicture}
\draw (0,0) circle [radius=1cm];
\filldraw (1,0) circle [radius=1pt] node [right] {$\pi_{x_{\rho_0}}^{(1)}(y)$};
\filldraw (0,1) circle [radius=0pt] node [above right] {$\pi_{x_{\rho_0}}^{(2)}(z)$};
\filldraw (-1,0) circle [radius=0pt] node [left] {$\pi_{x_{\rho_0}}^{(2)}(v)$};
\filldraw (0,-1) circle [radius=1pt] node [below] {$\pi_{x_{\rho_0}}^{(1)}(u)$};
\draw (-1,-1.2) to (-1, 2.2);
\draw(-1.2,2.2) to (1.2,-.2);
\draw (1,0) to (-1,0) [dashed];
\draw (0,1) to (0,-1) [dashed];
\draw (0,-1) to (-1,0) [dashed];
\end{tikzpicture}
\caption{Sample situation of Equation (\ref{eq.proof Ck closed}). The dashed lines encode transversality}
\end{figure}

As a result, for any distinct triple  $x,y,z\in\dg$ we can  find $u,v\in \dg\backslash\{x\}$ such that $y,z,v,u$ are cyclically ordered and
\begin{align}\label{eq.proof Ck closed}
\pi_{x_{\rho_0}}^{(1)}(y)\tv \pi_{x_{\rho_0}}^{(2)}(v), \; \pi_{x_{\rho_0}}^{(1)}(u)\tv \pi_{x_{\rho_0}}^{(2)}(z),\; 
\pi_{x_{\rho_0}}^{(1)}(u)\tv \pi_{x_{\rho_0}}^{(2)}(v).
\end{align}
Indeed we find $u$ in the interval\footnote{Recall from Notation \ref{n.interval} that $(u,x)_y$ is the interval in $\dg\backslash\{u,x\}$ not containing $y$} $(y,x)_z$ such that $\pi_{x_{\rho_0}}^{(1)}(u)\tv \pi_{x_{\rho_0}}^{(2)}(z)$ as otherwise this would contradict Equation \eqref{eq.proof Ck closed0}. Since transversality is an open condition, Equation \eqref{eq.proof Ck closed0} implies that we find an open set $\hat{U}\subset (u,x)_y$ such that for all $w\in\hat{U}$ we have $\pi_{x_{\rho_0}}^{(1)}(y)\tv \pi_{x_{\rho_0}}^{(2)}(w)$. Moreover Equation \eqref{eq.proof Ck closed0} implies also that the set of points in the interval $w\in (u,x)_y$ such that $\pi_{x_{\rho_0}}^{(1)}(u)\tv\pi_{x_{\rho_0}}^{(2)}(w)$ is dense in $\hat{U}$. Therefore we find the $v\in\hat{U}$ as desired.

Set $X_{\r_n}:=\quotient{ x^{d-k+1}_{\r_n}}{x^{d-k-2}_{\r_n}}$. Proposition \ref{prop.hyperconvex proj to RP2} gives that, for fixed $n$, the  curve $ \pi_{x_{\rho_n}}^{(1)}:\dg\to \P(X_{\r_n})$ is 
 hyperconvex; as a result, for all cyclically ordered quadruples $a,b,c,d\in\dg\backslash\{x\}$,
$$\gc^{X_{\r_n}}_1(\pi_{x_{\rho_n}}^{(1)}(a),\pi_{x_{\rho_n}}^{(2)}(b),\pi_{x_{\rho_n}}^{(2)}(c),\pi_{x_{\rho_n}}^{(1)}(d))>1.$$ 
 Here $\gc^{X_{\r_n}}_1$ is the cross ratio of points and hyperplanes of ${X_{\r_n}}$ as in Definition \ref{def.k cross ratio}.

Furthermore, it follows from the continuity of $\gc_{d-k-1}$ that if $y,z,v,u$ satisfy Equation \eqref{eq.proof Ck closed}, we have 
\begin{align}\label{eq.proof Ck closed-2}
&\gc^{X_{\r_n}}_1(\pi_{x_{\rho_n}}^{(1)}(y),\pi_{x_{\rho_n}}^{(2)}(z),\pi_{x_{\rho_n}}^{(2)}(v),\pi_{x_{\rho_n}}^{(1)}(u))\\
\to &\gc^{X_{\r_0}}_1(\pi_{x_{\rho_0}}^{(1)}(y),\pi_{x_{\rho_0}}^{(2)}(z),\pi_{x_{\rho_0}}^{(2)}(v),\pi_{x_{\rho_0}}^{(1)}(u))\nonumber
\end{align}
as $n\to \infty$. Indeed we can choose  $V^{k-1}\in \Gr_{d-k-1}(E)$ transverse to $x_{\r_n}^{d-k+1}$ for all $n$ big enough such that (see Proposition \ref{prop.k-cross ratio and projection})
$$\gc^{X_{\r_n}}_1(\pi_{x_{\rho_n}}^{(1)}(y),\pi_{x_{\rho_n}}^{(2)}(z),\pi_{x_{\rho_n}}^{(2)}(v),\pi_{x_{\rho_n}}^{(1)}(u))=\gc_{d-k-1}\left(A_n,B_n,C_n,D_n \right)$$
for
\begin{align*}
&A_n:=(a_{\r_n}^k\cap x_{\r_n}^{d-k+1})\oplus x_{\r_n}^{d-k-2},\quad B_n:=(b_{\r_n}^{k+1}\cap x_{\r_n}^{d-k+1})\oplus V^{k-1},\\
&C_n:=(c_{\r_n}^{k+1}\cap x_{\r_n}^{d-k+1})\oplus V^{k-1},\quad D_n:=(d_{\r_n}^k\cap x_{\r_n}^{d-k+1})\oplus x_{\r_n}^{d-k-2}.
\end{align*}

Assume now that $\r_0$ does not satisfy property $C_k$, let $x,y,z\in\dg$ be pairwise distinct points violating property $C_k$ for $\r_0$, and choose $u,v$ satisfying Equation  \eqref{eq.proof Ck closed}. The failure of transversality implies that 
$$\gc^{X_{\r_0}}_1(\pi_{x_{\rho_0}}^{(1)}(y),\pi_{x_{\rho_0}}^{(2)}(z),\pi_{x_{\rho_0}}^{(2)}(v),\pi_{x_{\rho_0}}^{(1)}(u))=\infty.$$

Since transversality is an open condition,  there is $w_0\in\dg$ such that $y,w_0,z,v,u\in\dg\backslash\{0\}$ are in that cyclic order and $\pi_{x_{\rho_0}}^{(1)}(w)\tv \pi_{x_{\rho_0}}^{(2)}(v)$ for all $w\in(w_0,y)_z$, so that $\gc^{X_{\r_0}}_1(\pi_{x_{\rho_0}}^{(1)}(w),\pi_{x_{\rho_0}}^{(2)}(z),\pi_{x_{\rho_0}}^{(2)}(v),\pi_{x_{\rho_0}}^{(1)}(u))$
is defined for all such $w$. Then the monotonicity of the cross ratio associated to $\pi_{x_{\rho_n}}$, implying that 
\begin{align*}
&\gc^{X_{\r_n}}_1(\pi_{x_{\rho_n}}^{(1)}(w),\pi_{x_{\rho_n}}^{(2)}(z),\pi_{x_{\rho_n}}^{(2)}(v),\pi_{x_{\rho_n}}^{(1)}(u))\\
>&\gc^{X_{\r_n}}_1(\pi_{x_{\rho_n}}^{(1)}(y),\pi_{x_{\rho_n}}^{(2)}(z),\pi_{x_{\rho_n}}^{(2)}(v),\pi_{x_{\rho_n}}^{(1)}(u))
\end{align*}
for all $w\in(w_0,y)_z$, implies that also 
$$\gc^{X_{\r_0}}_1(\pi_{x_{\rho_0}}^{(1)}(w),\pi_{x_{\rho_0}}^{(2)}(z),\pi_{x_{\rho_0}}^{(2)}(v),\pi_{x_{\rho_0}}^{(1)}(u))=\infty.$$
which in turn yields that property $C_k$ is violated on the whole open interval  $(w_0,y)_z\subset \dg$, a contradiction to strong irreducibility.
\end{proof}

\subsection{Reducible representations and Fuchsian locii}
We will now discuss properties $H_k$ and $C_k$ for reducible representations. The following is well known and easy to check:
\begin{lem}
Let $\rho:\G\to\SL(L_1\oplus L_2)$ be reducible. If $\rho$ is $k$-Anosov, then the dimension $k_i$ of the intersection $x_i^{k_i}:=x^{k}\cap L_i$ is constant, $x^k=x_1^{k_1}\oplus x_2^{k_2}$, the restriction $\rho|_{L_i}$ is $k_i$-Anosov, and $x\mapsto x_i^{k_i}$ is the associated boundary map. 
\end{lem}
\begin{proof}
First note that, since the subspaces $L_i$ are invariant, for all $\g\in \G$ the subspace $\g^k_+$ splits as the direct sum $(\g^k_+\cap L_1) \oplus (\g^k_+\cap L_2)$. As the functions $y\mapsto \dim(y^k\cap L_i)$ are both upper-semicontinuous and the set of fixed points is dense in $\dg$, it follows that those maps are constant. This yields the decomposition $x^k=x_1^{k_1}\oplus x_2^{k_2}$ everywhere. Since $\xi^k$ is dynamics preserving, we get that the maps $x\mapsto x_i^{k_i}$ are also dynamics preserving. Moreover it follows easily that $k_i-$th root gap of the representation $\rho|_{L_i}$ is not smaller than the $k-$th root gap of $\rho$. Thus $\rho|_{L_i}$ is $k_i$-Anosov with boundary map $x\mapsto x_i^{k_i}$.
\end{proof}
\begin{prop}\label{p.red1}
Let $\rho:\G\to\SL(L_1\oplus L_2)$ be $\{k-1,k,k+1\}-$Anosov and reducible. Assume that $(k-1)_1=k_1-1$. Then $\rho$ has property $H_k$ if and only if $(k+1)_1=k_1+1$ and $\rho|_{L_1}$ has property $H_{k_1}$. 
\end{prop}
\begin{proof}
Let $d_1=\dim L_1$.
Observe that, under our assumption $x^k\cap z^{d-k+1}=x_1^{k_1}\cap z_1^{d_1-k_1+1}\subset L_1$. As a result the sum 
$$x^{k-1}+\left(x^k\cap z^{d-k+1}+y^k\cap z^{d-k+1}\right)+z^{d-k-1}$$
can only be direct if $L_2\subset x^{k-1}+ z^{d-k-1}$ and $\rho|_{L_1}$ has property $H_{k_1}$. The converse implication is clear. 
\end{proof}
The analogue statement for property $C_k$ is proven in the same way. We state it for future reference:
\begin{prop}\label{p.red2}
Let $\rho:\G\to\SL(L_1\oplus L_2)$ be reducible and $\{k-1,k,k+1,k+2\}-$Anosov. Assume that $(k-1)_1=k_1-1$. If $(k+1)_1=k_1+1$,  $(k+2)_1=k_1+2$ and $\rho|_{L_1}$ has property $C_{k_1}$, then $\rho$ has property $C_k$ . \end{prop}

Observe that if $(k-1)_1\neq k_1-1$, then necessarily $(k-1)_2=k_2-1$, so Proposition \ref{p.red1} and Proposition \ref{p.red2} can be applied to $L_2$.

An easy way to obtain many examples of Anosov representations is to deform representations in the so-called Fuchsian loci, the set of representations obtained as composition of holonomies of hyperbolizations with $\SL(\R^2)$-representations. We fix here the notation for such representations:
\begin{example}[Fuchsian Loci]\label{ex 	hyperconvex reps}
Denote by $\tau_{d_i}:\SL(\R^2)\to \SL(\R^{d_i})$ the $d_i-$dimensional irreducible representation\footnote{this is uniquely defined up to conjugation} of $\SL(\R^2)$ and set $$\tau_{(d_1,\ldots,d_j)}:=\tau_{d_1}\oplus \ldots\oplus \tau_{d_j}$$ for positive  integers $d_1\geq d_2 \geq \ldots\geq d_j$ with $d=d_1+\ldots+d_j$.
Let $\r_{hyp}:\G\to \SL(\R^2)$ be a discrete and faithful representation. We call the set of representations obtained as composition $\tau_{\un d}\circ \r_{hyp}$ as $\r_{hyp}$ varies in the Teichm\"uller space a \emph{Fuchsian locus}, or the $\un d$-Fuchsian locus for the specified multi-index $\un d=(d_1,\ldots,d_j)$. 
\end{example}

It is easy to verify that the irreducible representation satisfies property $H_k, C_k$ for all $k$, as in this case the equivariant boundary map is the well studied Veronese curve. As a result we obtain:
\begin{cor}
In the notation of Example \ref{ex 	hyperconvex reps}, a representation of the form $\rho=\tau_{(d_1,\ldots, d_m)}\circ \r_{hyp}$ has property $H_k$ if and only if $d_1-d_2>2k$; if $d_1-d_2>2k+2$ the representation additionally has property $C_k$ .
\end{cor}

\subsection{$\Theta$-positive representations}
Another (conjectural) class of higher rank Teichm\"uller theories are the so-called $\Theta$-positive representations, as introduced by Guichard and Wienhard \cite{GWpositivity}. For the purposes of this paper we will only be concerned with $\Theta-$positive representations into $\SO(p,q)$ where we assume $q>p$. When considering $\SO(p,q)$ as a subgroup of $\SL(p+q,\R)$, $\Theta$-positive representations are (conjecturally) $\{1,\ldots,p-1\}$-Anosov.\footnote{From now on, when we write '$\Theta-$positive representations into $\SO(p,q)$' we additionally assume that they are $\{1,\ldots,p-1\}$-Anosov. Conjecturally Anosovness already follows from positivity  \cite[Conjecture 5.4]{GWpositivity}.}

We will not need the precise definition of $\Theta$-positive representations, and as it would require introducing many concepts from Lie theory which are not important for our paper, we refer to \cite[Section 4.5]{GWpositivity}; the only important property of such representations that will be relevant here is a precise form of positivity of the associated boundary map which we now recall. For this we choose a basis of $\SL(p+q,\R)$ such that the quadratic form $Q$ of signature $p,q$ is represented by the matrix 
$$\bpm 0&0&K\\0&J&0\\K^t&0&0\epm$$ where 
$$K=\begin{pmatrix}0&0&(-1)^{p-1}\\ 0&\iddots&0\\-1&0&0\end{pmatrix}\text{ and } J=\begin{pmatrix}0&0&1\\ 0&-\Id_{q-p}&0\\1&0&0\end{pmatrix}.$$
In this section we will denote by $\calF$ the subset of the partial flag manifold associated to $\SL(p+q,\R)$ consisting of subspaces of dimension $\{1,\ldots, p-1,q+1,\ldots,q+p-1\}$ such that the first $p-1$ subspaces are isotropic for $Q$ and the others are their orthogonals with respect to $Q$.
We will furthermore denote by $Z$ and $X$ the partial flags in $\calF$ such that $Z^l=\langle e_1,\ldots, e_l\rangle$ and $X^l=\langle e_{d},\ldots, e_{d-l+1}\rangle$, in particular $X^k\cap Z^{d-k+1}=e_{d-k+1}$. Here, as above, $l$ ranges in the set $\{1,\ldots, p-1, q+1,\ldots, q+p\}$.

Given a positive real number $v$, and an integer $1\leq k\leq p-2$ we define $E_{k}(v)$ as the matrix that differs from the identity only in the positions $(k,k+1)$ and $(d-k, d-k+1)$ where it is equal to $v$. Instead, for $k=p-1$ we choose a vector $ v\in \R^{q-p+2}$ which is positive for the quadratic form associated to $J$ and has positive first entry. For each such $v$ we define 
$$E_{p-1}(v)=\bpm 
\Id_{p-2}&0&0&0&0\\
0&1& v^t&q_J(v)&0\\
0&0&\Id_{q-{p+2}}&J v&0\\
0&0&0&1&0\\
0&0&0&0&\Id_{p-2}
\epm.$$

In order to define the properties of $\Theta$-positive representations that we will need, we will use the following reduced expression of the longest element $w_0$ of the Weyl group $W(\Theta)=W_{B_{p-1}}$, i.e. the Weyl group associated to the root system $B_{p-1}$: Let $S$ be a standard generating set in standard order of $W_{B_{p-1}}$, i.e. we write $S= \{s_1,\ldots,s_{p-1}\}$, where $s_{p-1}$ corresponds to reflection along the only long root in a set of simple roots. Let $S_e\subset S$ be the elements with even index and $S_o\subset S$ the elements with odd index. 
Denote the product of all the elements of $S_e$ by $a$ and the product of all elements of $S_o$ by $b$. Then $w_0$ can be expressed as $w_0=(a b)^{\frac h2}$, where $h$ is the Coxeter number, \cite[pp.150-151]{Bou2} (see also \cite[Lemma 4.3]{DS}). Note that for the root system of type $B_{p-1}$, $p\geq 3$ the Coxeter number is always even.

We now consider the unipotent subgroup $U_{\Theta}$ of the stabilizer in $\SO(p,q)$ of the partial flag $Z$; our next goal is to define its \emph{positive semigroup} $U_\Theta^{>0}$  (cfr. \cite[Theorem 4.5]{GWpositivity}). We denote by $c_J(\R^{q-p+2})\subset \R^{q-p+2}$ the set of vectors that are positive for the quadratic form associated to $J$ and have positive first entry. Then we set
$$V_{\Theta}:=\left.\{ \ov v=(v_1,\ldots,v_{p-2},v_{p-1})^t\in \R^q\right| v_1,\ldots, v_{p-2}\in \R_{>0}, v_{p-1} \in c_J(\R^{q-p+2})\}.$$ 
Given $\ov v\in V_{\Theta}$ we set 
$$ab(\ov v)=\left(\prod_{j\leq p-1,\: j \text{ even}} E_j(v_j)\right)\cdot \left(\prod_{j\leq p-1,\: j \text{ odd}} E_j(v_j)\right)$$
For $\ov v_1,\ldots \ov v_{\frac{h}{2}}\in V_{\Theta}$ we define the \emph{positive element} $P(\ov v_1,\ldots \ov v_{\frac{h}{2}})$ as the product
$$P(\ov v_1,\ldots \ov v_{\frac{h}{2}})=ab(\ov v_1)\ldots ab(\ov v_\frac{h}{2})$$
The positive semigroup $U_\Theta^{>0}$ consists precisely of the positive elements defined above (cfr. \cite[Theorem 4.5]{GWpositivity}). This allows to recall the notion of positivity for triples of flags associated to $\SO(p,q)$ (cfr. \cite[Definition 4.6]{GWpositivity}):  a triple $(A,B,C)\in \calF^3$ is $\Theta$-\emph{positive} if there exists an element $g\in\SO_0(p,q)$ and a positive element $P(\ov v_1,\ldots \ov v_{\frac{h}{2}})$ such that 
$$(gA,gB,gC)=(X, P(\ov v_1,\ldots \ov v_{\frac{h}{2}})X,Z).$$

\begin{defn}[{\cite[Definition 5.3]{GWpositivity}}]
A representation $\rho:\G\to\SO(p,q)$ is \emph{$\Theta$-positive} if and only if it admits a positive equivariant boundary map $\xi:\partial \G\to \calF$; that is, for every positively oriented triple $(z,y,x)\in\partial \G$, the triple $(\xi(z),\xi(y),\xi(x))$ is positive. 
 \end{defn}

It was proven in \cite[Theorem 10.1]{PSWB} that $\Theta$-positive representations $\rho:\G\to \SO(p,q)$ have property $H_k$ for $1\leq k<p-2$. They also satisfy property $C_k$ in a slightly smaller range: 
\begin{prop}\label{p.c1}
Let $\rho:\G\to \SO(p,q)$ be $\Theta$-positive. For every $1\leq k\leq p-3$ 
the representation $\rho$ has property $C_k$.
\end{prop}
\begin{proof}
We set $d=p+q$. In order to verify that the representation $\rho$ has property $C_k$, it is enough to verify that, for every positively and for every negatively oriented triple $(x,y,z)$ the sum 
$z^{d-k-2} + (z^{d-k+1}\cap y^{k}) + x^{k+1}$ is direct. 

Since this last property is invariant by the $\SL(p+q,\R)$-action, and thus in particular by the $\SO_0(p,q)$-action, it is enough to verify that for each positive element $P(\ov v_1,\ldots \ov v_{\frac{h}{2}})$ the triple $(X, P(\ov v_1,\ldots \ov v_{\frac{h}{2}})X,Z)$ is  such that the sum
$$Z^{d-k-2} + (Z^{d-k+1}\cap P(\ov v_1,\ldots \ov v_{\frac{h}{2}})X^{k}) + X^{k+1}$$
is direct, and the analogue result for the triple $(X, P(\ov v_1,\ldots \ov v_{\frac{h}{2}})^{-1}X,Z)$.

In turn this is equivalent to verify that for every admissible choice of $\ov v_1,\ldots \ov v_{\frac{h}{2}}\in V_{\Theta}$ and every $1\leq k\leq p-3$ the coefficient in position $(d-k-1,d-k+1)$ of the matrices $P(\ov v_1,\ldots \ov v_{\frac{h}{2}})$ and $P(\ov v_1,\ldots \ov v_{\frac{h}{2}})^{-1}$ doesn't vanish.

This follows readily from the definitions: indeed, given unipotent matrices $A,B$ it holds 
$$
\begin{array}{ll}
(AB)_{d-k-1,d-k}&= A_{d-k-1,d-k}+B_{d-k-1,d-k},\\
(AB)_{d-k,d-k+1}&= A_{d-k,d-k+1}+B_{d-k,d-k+1},\\
(AB)_{d-k-1,d-k+1}&= A_{d-k-1,d-k+1}+A_{d-k-1,d-k}B_{d-k,d-k+1}+B_{d-k-1,d-k+1}.
\end{array}$$
In particular, if $ab(\ov v_i), i=1,\ldots\frac{h}{2}$ is the matrix introduced in the definition of positive elements, it holds
$$
\begin{array}{rl}
(ab(\ov v_i))_{d-k-1,d-k}&= v_{k-1},\\
(ab(\ov v_i))_{d-k,d-k+1}&= v_k,\\
\end{array}$$
and
$$
\begin{array}{rl}
(ab(\ov v_i)^{-1})_{d-k-1,d-k}&= -v_{k-1},\\
(ab(\ov v_i)^{-1})_{d-k,d-k+1}&= -v_k,\\
\end{array}$$

One readily checks by induction, using the fact that $h\geq 2$, that the relevant coefficients don't vanish: indeed it is sum of positive numbers of which at least one is non-zero. To be more precise, in the case of $P(\ov v_1,\ldots \ov v_{\frac{h}{2}})$ all coefficients in positions $(d-k-1,d-k)$, $(d-k,d-k+1)$ and $(d-k-1,d-k+1)$ in all the matrices whose product gives  $P(\ov v_1,\ldots \ov v_{\frac{h}{2}})$ are either positive or zero; in the case  of $P(\ov v_1,\ldots \ov v_{\frac{h}{2}})^{-1}$ all coefficients in positions $(d-k-1,d-k)$, $(d-k,d-k+1)$  in all the matrices whose product gives  $P(\ov v_1,\ldots \ov v_{\frac{h}{2}})^{-1}$ are negative or zero and  thus all the coefficients in position  $(d-k-1,d-k+1)$ are positive because they are sums of products of pairs of the previous coefficients. 
\end{proof}

\section{Positively ratioed representations}\label{sec:posrat}
{ We now turn to the study of positively ratioed representations. The following definition is essentially due to Martone-Zhang (\cite{MZ}):

\begin{defn} 
We say that a $k-$Anosov representation $\r:\G\to \SL(E)$ is \emph{$k-$positively ratioed} if for all cyclically ordered quadruples $(x,y,z,w)$ of points in $\dg$ 
\begin{align*}
\gc_k(x^k,y^{d-k},z^{d-k},w^k)>1.
\end{align*}
\end{defn}

\begin{remark}
Our definition of $k$--positively ratioed representation is slightly stronger than the one in \cite{MZ}:
A representation is positively ratioed in the sense of Martone Zhang if it is a $k-$Anosov representation $\r:\G\to\SL(E)$ such that
\begin{align*}
\log |\gc_k(x^k,y^{d-k},z^{d-k},w^k)\gc_{d-k}(x^{d-k},y^{k},z^{k},w^{d-k})|>0.
\end{align*}
To see that a $k-$positively ratioed representation is positively ratioed in the sense of Martone-Zhang, it is enough to observe that
$$\gc_{d-k}(x^{d-k},y^{k},z^{k},w^{d-k})=\gc_{k}(y^k,x^{d-k},w^{d-k},z^k),$$
as well as that $(y,x,w,z)$ are cyclically ordered in $\dg$ if $(x,y,z,w)$ are.
\end{remark}

Representations satisfying properties $H_k$, $H_{d-k}$ are $k$-positively ratioed:
\begin{thm}\label{prop.hyperconvex implies positively ratioed}
If $\r:\G\to\SL(E)$ satisfies properties $H_k$ and $H_{d-k}$, then $\r$ is $k-$positively ratioed.
\end{thm}

\begin{proof} Note that property $H_{d-k}$ implies  that $\xi^{d-k}$ has $C^1-$image (Proposition \ref{pro 1-hyperconvex for property Hk}). Let $\Phi_y\in T_{y^{d-k}}\xi^{d-k}\backslash\{0\}$.
We claim that
$$d_{y^{d-k}} \gc_k(x^k,y^{d-k},\cdot\,,w^k)(\Phi_y)\neq 0$$ 
for all pairwise distinct $x,y,w\in\dg$:
Fix such $x,y,w$. We find a basis $(e_1,\ldots,e_d)$ of $E$ such that 
\begin{align*}
y^j=&\langle e_1,\ldots, e_j\rangle, \quad j=d-k-1,d-k,d-k+1.
\end{align*} 
According to Proposition \ref{pro 1-hyperconvex for property Hk} the curve 
$$y^{d-k}_t:=\langle e_1,\ldots, e_{d-k-1},e_{d-k}+te_{d-k+1}\rangle, \quad t\in (-\epsilon,\epsilon)$$ 
is tangent to $T_{y^{d-k}}\xi^{d-k}$ at $y^{d-k}=y_0^{d-k}$, therefore
$$ d_{y^{d-k}} \gc_k(x^k,y^{d-k},\cdot,w^k)(\Phi_y)\neq 0 \Longleftrightarrow \left.\frac{d}{dt}\right|_{t=0} \gc_k(x^k,y^{d-k},y_t^{d-k},w^k)\neq 0.$$
Set $Y:=\quotient{y^{d-k+1}}{y^{d-k-1}}$ and recall from Definition \ref{def.Px} the $k$-th $\R\P^1$-projection
$$P_y=\dg\to\Pp(\quotient{y^{d-k+1}}{y^{d-k-1}}).$$
Corollary \ref{cor.projective and k-cross ratio} yields that
$$\gc_k(x^k,y^{d-k},y_t^{d-k},w^k)=\pc (P_y(x),P_y(y),[y_t^{d-k}]_Y,P_y(w)).$$
Property $H_{k}$ guarantees that $P_y(x),P_y(w),P_y(y)$ are pairwise distinct (Proposition \ref{prop.P projec injective}). Moreover the derivative of the $C^1-$curve $[y_t]_Y$ is not $0$, as $\Pp(\quotient{y^{d-k+1}}{y^{d-k-1}})\simeq \Pp(\langle e_{d-k},e_{d-k+1}\rangle)$ and $[y_t]_Y$ is the projectivization of the path $t\mapsto e_{d-k}+te_{d-k+1}$. Therefore it follows from Lemma \ref{lem.non-zero derivative of projective cross ratio} that 
\begin{align*}
&\left.\frac{d}{dt}\right|_{t=0} \pc(P_y(x),P_y(y),[y_t^{d-k}]_Y,P_y(w))\neq 0
\end{align*}
which is equivalent to $\left.\frac{d}{dt}\right|_{t=0} \gc_k(x^k,y^{d-k},y_t^{d-k},w^k)\neq 0$ and proves the claim.

Now, let $y\to \Phi_y\in T_{y^{d-k}}\xi^{d-k}\backslash\{0\}$ be a continuous map defined on a connected component of $\dg\backslash \{x,w\}$, choose $\Phi_y$ so that $(x,y,y_t,w)$ are cyclically ordered for a curve  $y^{d-k}_t$ with derivative $\Phi_y$ and $t>0$.

The regularity of the cross ratio implies that the never vanishing map
$$y\mapsto d_{y^{d-k}} \gc_k(x^k,y^{d-k},\cdot,w^k)(\Phi_y)$$ 
is continuous on connected components of $\dg\backslash \{x,w\}$, its sign is thus  constant on each component. The cocycle identity,  Lemma \ref{lem property of grassmannian cro} $(3)$, implies that for cyclically ordered $x,y,z,w\in\dg$, the cross ratio $\gc_k(x,y,z,w)$ is smaller than one if  the sign of the derivative is  negative, and is bigger than 1 if it is positive. Since, for $z_n\to w$, we deduce from  the continuity of the cross ratio and  Lemma \ref{lem property of grassmannian cro} $(3)$ that $\gc_k(x,y,z_n,w)\to \infty$, it  cannot be $\gc_k(x,y,z,w)<1$ for all cyclically ordered $x,y,z,w\in\dg$ and thus $\gc_k(x,y,z,w)>1$.
\end{proof}

\section{Proof of the collar lemma}\label{s.collar}
\addtocontents{toc}{\protect\setcounter{tocdepth}{1}}
Given non-trivial elements $g,h\in \G$, we denote, as usual, by $g_{\pm},h_{\pm}\in \dg$ the respective attractive and repulsive fix points, and we call the pair $g,h\in\G\backslash\{e\}$ \emph{linked} if $(g_{-},h_{-},g_{+},h_{+})$ are cyclically ordered - for $\G=\pi_1(S_g)$ this holds if and only if the corresponding closed geodesics for some (and thus any) choice of hyperbolic metric intersect in $S_g$. Note that this in particular asks that the four points are distinct. Clearly $g,h\in \G$ are linked if and only if $g,h^{-1}\in \G$ are linked. Throughout the section $g,h\in\G$ will always denote a linked pair.

 \begin{figure}[h]
\begin{tikzpicture}
\draw (0,0) circle [radius =1];
\draw (-1,0) to (1,0);
\draw (0,-1) to (0,1);
\node at (.5,0)  {$>$};
\node at (0,.5) [rotate=90] {$>$};
\node at (1,0) [right] {$g_+$};
\node at (-1,0) [left] {$g_-$};
\node at (0,1) [above] {$h_+$};
\node at (0,-1) [below] {$h_-$};
\node at (.8,.8)[right]{$hg_+$};
\filldraw (.7,.7) circle [radius=1pt];
\end{tikzpicture}
\caption{The relative positions of the fixed points of linked $g,h\in \G$.}
\end{figure}
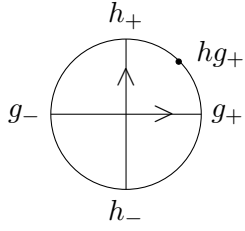

This section is devoted to the proof of Theorem \ref{thm.INTROcollar lemma} from the Introduction, which we restate here for the reader's convenience:

\begin{thm}\label{thm.collar lemma}
If  $\r:\G\to \SL(E)$  satisfies properties $H_{k-1}$, $H_{k}$, $H_{k+1}$, $H_{d-k-1}$, $H_{d-k}$, $ H_{d-k+1}$ and $C_{k-1}$, $C_k$. Then for every linked pair $g,h\in \G$
\begin{equation}\label{e.collar}
\frac{\lambda_1\ldots\lambda_k}{\lambda_d\ldots\lambda_{d-k+1}}(g_{\rho})>\left(1-\frac{\lambda_{k+1}}{\lambda_k}(h_{\rho})\right)^{-1}.
\end{equation}
\end{thm}

The proof is divided in three steps: In the first step we use the cross ratio $\pc_{h_-^{d-k-1}}$ and the connection to the cross ratio $\gc_{d-k}$ to bound the right hand side of (\ref{e.collar}) from above by 
$$\gc_{d-k} (h_-^{d-k},h_+^k,gh_+^k,\left(g_+^k\cap h_-^{d-k+1}\right)\oplus h_-^{d-k-1} ).$$
The main step of the proof is the second step, where we show that moving $\left(g_+^k\cap h_-^{d-k+1}\right)\oplus h_-^{d-k-1}$ to $g_+^{d-k}$ does not decrease the value of $\gc_{d-k}$. In the third step we use the fact that the representation is positively ratioed to further increase the value by replacing $h_-^{d-k}$ with $g_-^{d-k}$.

\subsection*{Step 1: Relating the eigenvalue gap with a Grassmannian cross ratio}
In this section we will consider the two dimensional vector space $H:=\quotient{h_{-}^{d-k+1}}{h_{-}^{d-k-1}}$ and study the $k$-th $\R\P^1$ projection based at $h_-$ 
$$P_{h_{-}}:\dg\to \Pp\left(\quotient{h_{-}^{d-k+1}}{h_{-}^{d-k-1}}\right)$$ introduced in Definition \ref{def.Px}. 
\begin{prop}\label{prop.col.inequality using properties of the projective cross ratio}
Assume that $\rho$ has property $H_k$. Then 
$$\left(1-\frac{\lambda_{k+1}}{\lambda_k}(h_{\rho})\right)^{-1}<
\pc\left(P_{h_{-}}(h_-),P_{h_{-}}(h_+),P_{h_{-}}( gh_+), P_{h_{-}}(g_+) \right).$$
\end{prop}

\begin{proof}
Since $h_\r$ induces a linear map of the two dimensional vector space $H$ with eigenvalues $\lambda_k(h_{\r}), \lambda_{k+1}(h_{\r})$, we deduce from Proposition \ref{p.gc1} that 
$$\frac{\lambda_k(h_{\r})}{\lambda_{k+1}(h_{\r})}=
\pc\left(P_{h_{-}}(h_-),P_{h_{-}}(g_+), P_{h_{-}}(hg_+), P_{h_{-}}(h_+)\right).$$
Applying the symmetries of the projective cross ratio, namely $(1)$, $(8)$ and again $(1)$ from Lemma \ref{lem.symmetry of projective cro}, we get
$$\left(1-\frac{\lambda_{k+1}}{\lambda_k}(h_{\rho})\right)^{-1}=
\pc\left(P_{h_{-}}(h_-),P_{h_{-}}(h_+),P_{h_{-}}( hg_+), P_{h_{-}}(g_+) \right).$$

Since by hyperbolic dynamics $h_-,h_+,hg_+,gh_+,g_+$ are in that cyclic order on $\dg$ (\cite[Lemma 2.2]{LZ}) and $P_{h_{-}}$ preserves the order (Proposition \ref{prop.properties of proj Px}), we get
\begin{align*}
&\pc\left(P_{h_{-}}(h_-),P_{h_{-}}(h_+),P_{h_{-}}( hg_+), P_{h_{-}}(g_+) \right)\\
< &\pc\left(P_{h_{-}}(h_-),P_{h_{-}}(h_+),P_{h_{-}}( gh_+), P_{h_{-}}(g_+) \right),
\end{align*}
which proves the claim.
\end{proof}

We conclude the first step with:

\begin{cor}\label{cor.root compared to grassmannian cro} Let $\r$ satisfy property $H_k$. Then
$$\begin{array}{l}
\left(1-\frac{\lambda_{k+1}}{\lambda_k}(h_{\rho})\right)^{-1}<
\gc_{d-k}\left(h_-^{d-k},h_+^k, gh_+^{k}, \left(g_+^k\cap h_-^{d-k+1}\right)\oplus h_-^{d-k-1} \right)
\end{array}
$$
\end{cor}
\begin{proof}
By the lemma above it is enough to show that
\begin{align*}
&\pc\left(P_{h_{-}}(h_-),P_{h_{-}}(h_+),P_{h_{-}}( gh_+), P_{h_{-}}(g_+) \right)\\
=&\gc_{d-k}\left(h_-^{d-k},h_+^k, gh_+^{k}, \left(g_+^k\cap h_-^{d-k+1}\right)\oplus h_-^{d-k-1} \right).
\end{align*}
This however is a consequence of the relation of the projective cross ratio on $\P(\quotient{h_-^{d-k+1}}{ h_-^{d-k-1}})$ and $\gc_{d-k}$ as in Corollary \ref{cor.projective and k-cross ratio}. 
\end{proof}
\subsection*{Step 2: Replacing $\left(g_+^k\cap h_-^{d-k+1}\right)\oplus h_-^{d-k-1}$ with $g_+^{d-k}$}
Recall that for pairwise distinct points $x,y,z\in\dg$ we denote by $(x,y)_z$ the connected component of $\dg\backslash \{x,y\}$ that does \emph{not} contain $z$ (Notation \ref{n.interval}). We fix throughout the section an orientation on $S^1\simeq \dg$ such that the ordered triples $(h_-,x,g_+)$ for $x\in (h_-,g_+)_{g_-}$ are positively oriented.

As already in the proof of Theorem \ref{prop.hyperconvex implies positively ratioed}, instead of proving directly that replacing $\left(g_+^k\cap h_-^{d-k+1}\right)\oplus h_-^{d-k-1}$ with $g_+^{d-k}$ does not decrease the cross ratio, we will show that this is true infinitesimally, where hyperconvexity properties give us good control, and deduce the statement through the fundamental theorem of calculus.

More specifically we  consider the  map
$$\begin{array}{cccl}\eta:&(h_-,g_+)_{g_-}&\to &\Gr_{d-k}(E),\\
& x&\mapsto &\left(x^{d-k+1}\cap g_+^k\right)\oplus x^{d-k-1}.
\end{array}$$ 
This continuously interpolates  between $\left(g_+^k\cap h_-^{d-k+1}\right)\oplus h_-^{d-k-1}$ and $g_+^{d-k}$:
\begin{lem}\label{l.6.8}
Assume that $\rho$ satisfies property $H_k$. Then
$\eta(x)\to g_+^{d-k}$
for $x\to g_+$.
\end{lem}
\begin{proof}
This follows from Proposition \ref{prop:continuity}, as property $H_k$ guarantees that 
$$g^k_- \tv \left((x^{d-k+1}\cap g_+^k)\oplus x^{d-k-1}\right)$$
for all $x\in \dg\backslash\{g_{\pm}\}$.
\end{proof}
The goal of the second step is then to prove the following:
\begin{prop}\label{prop.step-2 cross ratio increases} If $\r:\G\to\SL(E)$ satisfies properties $H_{k-1},H_{k},H_{k+1}$, $H_{d-k-1},H_{d-k+1}$ and $C_{k-1},C_k$. Then
$$x\mapsto \gc_{d-k}\left(h_-^{d-k},h_+^k, gh_+^k, \eta(x) \right)$$ is non-decreasing for $x\in(h_-,g_+)_{g_-}$ moving towards $g_+$.
\end{prop}

As already mentioned we want to prove the infinitesimal version of Proposition \ref{prop.step-2 cross ratio increases}; however, while it follows from property $H_{d-k-1}$ and  $H_{d-k+1}$ respectively  that $\xi^{d-k-1}$ and $\xi^{d-k+1}$ have $C^1$ images, 
and the same is true for  the curve
 $$\begin{array}{cccl}
\xi^{d-k+1}_{g_+}:&\dg\backslash\{g_+\}&\to &\Pp(g_+^{k})\\ &x&\mapsto &x^{d-k+1}\cap g_+^{k},
\end{array}$$ 
there might not be a way to reparametrize $\dg$ to make both $\xi^{d-k-1}$ and $\xi^{d-k+1}_{g_+}$ simultaneously differentiable.  Consequently the image of $\eta$ might not be a $C^1$ submanifold of $\Gr_{d-k}(E)$.

\begin{remark}
If $k=1$, i.e. in the 'projective setting', the curve $x\mapsto \eta(x)= g_+^1\oplus x^{d-2}$ has already $C^1-$image. Thus  the proof simplifies slightly as we can directly prove Proposition \ref{prop.step-2 cross ratio increases}, with the same proof as Lemma \ref{lem.horizontal fol}.
\end{remark}

We overcome this difficulty by considering  $\eta$ as the diagonal in the surface $\Sigma\subset \Gr_{d-k}(E)$ which is the image of the map 
$$\begin{array}{cccl}
\xi_{g_+}^{d-k+1}\oplus \xi^{d-k-1}:& \left((h_-,g_+)_{g_-}\right)^2&\to &\Gr_{d-k}(E)\\ 
&(z,w)&\mapsto& \xi_{g_+}^{d-k+1}(z) \oplus\xi^{d-k-1}(w).
\end{array}$$
Since $\rho$ is $k$-Anosov such map is well defined, and since both $\xi_{g_+}^{d-k+1}$ and $\xi^{d-k-1}$ have $C^1$ image, the surface $\Sigma$ is indeed a $C^1$ submanifold of $\Gr_{d-k}(E)$. For a point $x\in (h_-,g_+)_{g_-}$, we will denote by $\Phi_x$ (resp. $\Psi_x$) in $T_{\eta(x)}\Gr_{d-k}(E)$ a chosen vector tangent to the horizontal (resp. vertical) leaf to the surface $\Sigma$. We will furthermore assume that the sign of these vectors is induced by the orientation of the interval $(h_-,g_+)_{g_-}$.

\begin{lem}\label{lem.col.combination of k-1,k+1 is lipschitz}
If $\r:\G\to\SL(E)$ satisfies properties $H_{d-k-1}$, $H_{d-k+1}$, then the image of $\eta$
is a Lipschitz submanifold of $\Gr_{d-k}(E)$, i.e. locally the graph of a Lipschitz map. Whenever defined, the derivative of $\eta$ is a non-negative linear combination of the vectors $\Phi_x$ and $\Psi_x$.
\end{lem}
\begin{proof}
If we choose smooth parametrizations $\varphi,\psi: [0,1]\to (h_-,g_+)_{g_-}$ such that 
$$\xi^{d-k-1}\circ \psi:[0,1]\to \Gr_{d-k-1}(E),\quad \xi_{g_+}^{d-k+1}\circ \varphi:[0,1]\to \Pp(g_+^{k})$$ are smooth maps, then 
$$\begin{array}{cccl}
F:&[0,1]^2&\to &\Gr_{d-k}(E)\\ 
&(s,t)&\mapsto& (\xi_{g_+}^{d-k+1}\circ \varphi(s)) \oplus(\xi^{d-k-1}\circ \psi (t)).
\end{array}$$
gives a $C^1$ parametrization of the surface $\Sigma$.

By construction $\varphi$ and $\psi$ are strictly monotone maps,  thus in particular invertible. As a result,  denoting by $\Delta\subset ((h_-,g_+)_{g_-})^2$ the diagonal, we have that  $D:= (\varphi^{-1},\psi^{-1})(\Delta)\subset[0,1]^2 $ is the graph  of a monotone map, in particular a Lipschitz submanifold. Since $\eta((h_-,g_+)_{g_-})=F(D)$ is the $C^1$ image of a Lipschitz submanifold, it has itself Lipschitz image. The claim about the derivative follows from the analogue claim on the curve $D\subset [0,1]^2$, which is the graph of a monotone map. 
\end{proof}
In particular it follows from Rademacher's theorem that $\eta$  has almost everywhere a well defined derivative, and the fundamental theorem of calculus applies. As a result Proposition \ref{prop.step-2 cross ratio increases}  follows as soon as the next two lemmas are established.

\begin{lem}\label{lem.vertical fol}
Let, as above, $\Phi_x$ denote the tangent to the horizontal leaf of  the surface $\Sigma$ at $\eta(x)$. Then properties $H_{k-1},H_k,H_{d-k+1}, C_{k-1}$ guarantee that
\begin{equation}\label{e.Phi}
d_{\eta(x)} \gc_{d-k} (h_-^{d-k},h_+^k, gh_+^k, \cdot) (\Phi(x))\geq 0.
\end{equation}
\end{lem}
\begin{proof}
In order to verify the claim it is enough to show that, for a smooth curve $x_t:(-\epsilon,\epsilon)\to \Gr_{d-k}(E)$ with $x_0=\eta(x)$, and $\dot x_0=\Phi_x$ it holds 
\begin{align}\label{e.Phi1}
\left.\frac{d}{dt}\right|_{t=0} \gc_{d-k} (h_-^{d-k},h_+^k, gh_+^k, x_t) \geq 0. 
\end{align}
We will prove this in two steps, first we will reduce the verification to checking monotonicity of a suitable projective cross ratio along a preferred path $x_t$, as expressed in Equation \eqref{eq.non vanishing cross ratio derivative}, then we will use hyperconvexity of the $(k-1)$-th $\R\P^2$-projection 
to check that Equation \eqref{eq.non vanishing cross ratio derivative} holds true.

Recall that $\Phi_x\in T_{\eta(x)}\Gr_{d-k}(E)$ is tangent to the horizontal leaf given by $\xi^{d-k+1}_{g_+}(w)\oplus x^{d-k-1}$, as $w$ varies in $(h_-,g_+)_{g_-}$. As a result,  we can assume that $x^{d-k-1}$ is always contained in $x_t$, furthermore since the tangent to  $\xi^{d-k+1}_{g_+}$ at $x$ is the tangent to the projective line $\P(x^{d-k+2}\cap g_+^k)$ (Proposition \ref{pro 1-hyperconvex for property Hk}), we can choose $x_t$ contained in  $x^{d-k+2}$.

This is helpful in reducing to a projective cross ratio: if we denote by $\ov X_g$ the space  $\ov X_g=(x^{d-k+2}\cap g_+^k)\oplus x^{d-k-1}$  and by $X_g$ its two dimensional quotient $X_g:=\quotient{\ov X_g}{x^{d-k-1}}$, and apply the cocycle identity of $\gc_{d-k}$ and Corollary \ref{cor.projective and k-cross ratio}, we get 
\begin{align}\label{eq.reduce to pcr}
&\gc_{d-k}\left(h_-^{d-k},h_+^k, gh_+^k, x_t \right)=\\
&\gc_{d-k}\left(h_-^{d-k},h_+^k, gh_+^k, x_0 \right)\gc_{d-k}\left(x_0,h_+^k, gh_+^k, x_t \right)=\nonumber \\
&\gc_{d-k}\left(h_-^{d-k},h_+^k, gh_+^k, x_0 \right) \pc ([x_0]_{X_g},[h_+^k\cap \overline X_g]_{X_g},[ gh_+^k\cap \overline X_g]_{X_g},[x_t]_{X_g}).\nonumber
\end{align}
Property $H_k$ implies that the quantity $ \gc_{d-k}\left(h_-^{d-k},h_+^k, gh_+^k,x_0\right)$ is always positive:  it is never zero, it is continuous in $x\in (h_-,g_+)_{g_-}$ and it is positive close to $h_-$ (Corollary \ref{cor.root compared to grassmannian cro}).
As a result, in order to prove the lemma, it  is enough to show that 
\begin{equation}\label{eq.non vanishing cross ratio derivative}
\left.\frac{d}{dt}\right|_{t=0} \pc ([x_0]_{X_g},[h_+^k\cap \overline X_g]_{X_g},[ gh_+^k\cap \overline X_g]_{X_g},[x_t]_{X_g})\geq  0.
\end{equation}

We will verify Equation \eqref{eq.non vanishing cross ratio derivative} showing that, for $t$ small, the four points  $([x_0]_{X_g},[h_+^k\cap \overline X_g]_{X_g},[ gh_+^k\cap \overline X_g]_{X_g},[x_t]_{X_g})$ are in this cyclic order on the line $\P(X_g)$. To this goal, set $X:=\quotient{x^{d-k+2}}{x^{d-k-1}}$ and recall that it follows from Proposition \ref{prop.hyperconvex proj to RP2} that, since $\r$ satisfies properties $H_{k-1},C_{k-1}$, the $(k-1)$-th $\R\P^2$ projection $\pi_x:\dg\to\calF(X)$ defines a hyperconvex curve.

\begin{figure}[h]
\begin{tikzpicture}
\begin{scope}
    \clip (-6,-1) rectangle (5,3);
\draw (0,0) circle [radius=2cm];
\filldraw (0,2) circle [radius=1pt] node[above right]{$\pi_x^{(1)}(g_+)$};
\draw (-2.9,2) [red, thick] to (2.5, 2);
\node at (-2.9, 2) [ left] {$X_g=\pi_x^{(2)}(g_+)$};
\end{scope}
\node at (-5, 3) {$\P X=\P\left(\quotient{x^{d-k+2}}{x^{d-k-1}}\right)$};

\filldraw (-2,0) circle [radius=1pt] node[left]{$\pi_x^{(1)}(h_+)$};
\draw (-2,-1) to (-2, 2.2);
\node at (-2, -1) [left] {$\pi_x^{(2)}(h_+)$};
\filldraw (-2,2) circle [radius=1pt];

\filldraw (2,0) circle [radius=1pt] node[right]{$\pi_x^{(1)}(x)$};
\draw (2,-1) to (2, 2.2);
\node at (2,-1) [right] {$\pi_x^{(2)}(x)$};
\filldraw (2,2) circle [radius=1pt];
\node at (2,1.7) [right] {$[x_0]$}; 

\draw (1.4,2.6) to (2,0);
\node at (1.4,2.6) [above] {$[x_t]$}; 
\filldraw (1.54,2) circle [radius=1pt];

\draw (2,0) to (-2,0);
\draw (2,0) to (0,2);

\begin{scope}
    \clip (0,-0.3) rectangle (2.3,2);
\draw[blue] (2,0) circle [radius=.5cm];
\end{scope}

\draw (-0.58,2.15) to (-2.6,0.4);
\node at (-3.3,0.45) [above] {$\pi_x^{(2)}(gh_+)$}; 
\filldraw (-0.75,2) circle [radius=1pt];
\end{tikzpicture}
\caption{Hyperconvexity of $\pi_x$ in an affine chart. We assume $t>0$, i.e. $x_t$ \emph{moves towards} $g_+$. The blue half circle identifies with $\P(\quotient{X}{ \pi_x^{(1)}(x)})$.}\label{fig:hypconvex-2}
\end{figure}
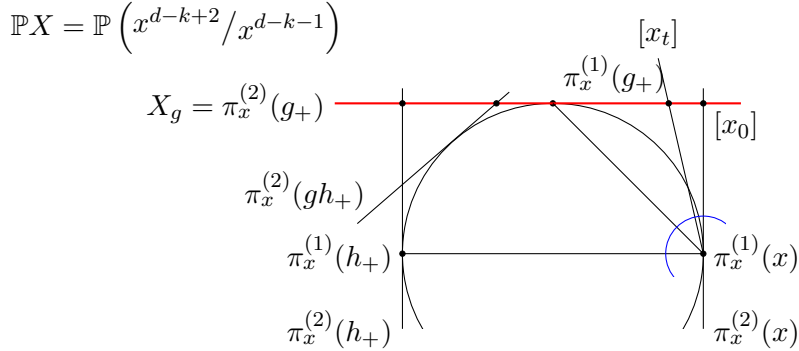
Note that $X_g=\pi_x^{(2)}(g_+)$ and thus $[x_0]_{X_g}=\pi_x^{(2)}(g_+)\cap \pi_x^{(2)}(x)$.
Then hyperconvexity of $\pi_X$ implies that, for $t$ small enough,  the four points
$$\left([\pi_x^{(2)}(gh_{+})\cap \pi_x^{(2)}(g_{+})],[\pi_x^{(2)}(h_{+})\cap \pi_x^{(2)}(g_{+})],[x_0], [\pi_x^{(1)}(g_{+})] \right)$$
are in that cyclic order on $\Pp(X_g)=\P(\pi_x^{(2)}(g^k_{+}))$ - cfr. Figure \ref{fig:hypconvex-2}. 

By assumption $\r$ satisfies properties $H_{k-1},H_{d-k+1}$, and thus is $(k-1)-$positively ratioed (Theorem \ref{prop.hyperconvex implies positively ratioed}); in particular for the linear approximation $L_x:(-\epsilon,\epsilon)\to \Gr_{d-k-1}(E)$ of $\xi^{d-k+1}$ at $x$ (recall Definition \ref{d.linearapp}) which is defined by $[x_t]_{X_g}= [L_x(t)]_{X}\cap \pi_x^{(2)}(g_+)$, we have 
$$\left.\frac{d}{dt}\right|_{t=0} \gc_{k-1}(h_+^{k-1},x^{d-k+1},L_x(t),g_+^{k-1})>0,$$
and thus $\left(\pi_x^{(1)}(h_+),\pi_x^{(2)}(x),[L_x(t)]_X,\pi_x^{(1)}(g_+)\right)$ for $t>0$ descend to points in that cyclic order on $\P(\quotient{X}{ \pi_x^{(1)}(x)})$ (cfr. Corollary \ref{cor.projective and k-cross ratio}). Using $ [x_t]_{X_g}= [L_x(t)]_{X}\cap \pi_x^{(2)}(g_+),$ and the hyperconvexity of $\pi_x$ we derive that, for all $t$ positive and small enough, 
$$\left([\pi_x^{(2)}(gh_{+})\cap \pi_x^{(2)}(g_{+})],[\pi_x^{(2)}(h_{+})\cap \pi_x^{(2)}(g_{+})],[x_0],[x_t], [\pi_x^{(1)}(g_{+})] \right)$$
are in that cyclic order on $\Pp(X_g)=\P(\pi_x^{(2)}(g_{+}))$ - cfr. again Figure \ref{fig:hypconvex-2}.

Thus for all $t$ positive and small enough, we can derive from the properties of the projective cross ratio (cfr.  Lemma \ref{lem.symmetry of projective cro}$(9)$) that
\begin{align*}
&\pc^{X_g} ([x_0],[h_+^k\cap \overline X_g],[ gh_+^k\cap\overline X_g],[x_t])>1\\
\Longrightarrow & \left.\frac{d}{dt}\right|_{t=0} \pc^{X_g} ([x_0],[h_+^k\cap\overline X_g],[ gh_+^k\cap\overline X_g],[x_t])\geq 0.\qedhere
\end{align*}
\end{proof}

\begin{lem}\label{lem.horizontal fol}
 Let, as above $\Psi_y$ denote the tangent to the vertical leaf of the surface $\Sigma$ at $\eta(y)$. Then properties $H_{k},H_{k+1},H_{d-k-1}, C_{k}$ guarantee that
$$d_{\eta(y)} \gc_{d-k} (h_-^{d-k},h_+^k, gh_+^k, \cdot) (\Psi_y)\geq 0.$$
\end{lem}

\begin{proof}
The proof is analogous to the proof of Lemma \ref{lem.vertical fol}. To avoid confusion with the proof of that lemma we write now $y\in (h_-,g_+)_{g_-}$. As above it is enough to show that  for a chosen curve $y_t:(-\epsilon,\epsilon)\to \Gr_{d-k}(E)$ with $y_0=\eta(y)$, $\dot y_0=\Psi_y$, 
\begin{align}\label{e.Phi1}
\left.\frac{d}{dt}\right|_{t=0} \gc_{d-k} (h_-^{d-k},h_+^k, gh_+^k, y_t) \geq 0. 
\end{align}
Observe that we can choose $y_t$ in the pencil\footnote{Recall Definition \ref{d.linearapp}} 
$\P\left( \quotient{y^{d-k+1}}{y^{d-k-2}\oplus \left(y^{d-k+1}\cap g_+^k\right)}\right)$:
indeed we can choose $y_t$ as the direct sum of the line $y^{d-k+1}\cap g_+^k$ and a linear approximation of $\xi^{d-k-1}$ at $y$.

If we set
\begin{align*}
Y&:= \quotient{y^{d-k+1}}{y^{d-k-2}}\\
Y_g&:= \quotient{y^{d-k+1}}{\left( (y^{d-k+1}\cap g_+^k)\oplus y^{d-k-2}\right)  } 
\end{align*}
we can use the same argument as in the first step of the proof of Lemma \ref{lem.vertical fol} to deduce that it is enough to show that
\begin{equation}\label{e.pcry}
\left.\frac{d}{dt}\right|_{t=0} \pc^{Y_g} ([y_0],[h_+^k\cap y^{d-k+1}],[ gh_+^k\cap y^{d-k+1}],[y_t]) \geq 0.
\end{equation}

Since $\r$ satisfies properties $H_{k},C_{k}$, the $k$-th $\R\P^2$-projection $\pi_y:\dg\to\calF(Y)$
is hyperconvex (Proposition \ref{prop.hyperconvex proj to RP2}). Since $\Pp(Y_g)\simeq\P(\quotient{Y}{ \pi_y^{(1)}(g_+)})$,  hyperconvexity implies that 
$$\left([\eta(y)]_{Y_g}=[\pi_y^{(1)}(y)]_{Y_g},[\pi_y^{(2)}(g_{+})]_{Y_g},[\pi_y^{(1)}(gh_{+})]_{Y_g}, [\pi_y^{(1)}(h_{+})]_{Y_g}\right)$$
are in that cyclic order on $\Pp(Y_g)\simeq\P\left(\quotient{Y}{ \pi_y^{(1)}(g_+)}\right)$ - compare Figure \ref{fig:hypconvex-proof2}.

 \begin{figure}[h]
\begin{tikzpicture}
\node at (-3.5,2.5)[above] {$\P(Y)=\P\left(\quotient{y^{d-k+1}}{y^{d-k-2}}\right)$};
\begin{scope}
    \clip (-2,-1.2) rectangle (2,2);
\draw (0,0) circle [radius=2cm];
\draw[blue] (0,2) circle [radius=.5cm];
\end{scope}

\filldraw (0,2) circle [radius=1pt] node[above]{$\pi_y^{(1)}(g_+)$};
\draw (-2,2) to (2.3,2);
\filldraw (2,2) circle [radius=1pt];

\filldraw (-1.91,-0.6) circle [radius=1pt] node[left]{$\pi_y^{(1)}(h_+)$};
\draw (-1.71,-1.2) to (-2.5,1);
\node at (-2.5,0.5) [left]{$\pi_y^{(2)}(h_+)$};

\filldraw (-1.5,1.3) circle [radius=1pt];
\node at (-1.5,1.5) [left]{$\pi_y^{(1)}(gh_+)$};
\node at (0.8,0.8) [below]{$[\eta(y)]$};

\draw (0,2) to (-1.5,1.3);
\draw (0,2) to (-1.91,-0.6);
\draw (0,2) to (2,0);

\filldraw (2,0) circle [radius=1pt] node[right]{$\pi_y^{(1)}(y)$};
\draw (2,-1) to (2,2.2);
\node at (2,2.2) [above]{$\pi_y^{(2)}(y)$};

\draw (0,2) to (2.4,0.3);
\node at (2.3,0.5) [right]{$[y_t]$};
\filldraw (2,0.58) circle [radius=1pt];

\end{tikzpicture}
\caption{Hyperconvexity of $\pi_Y$, where $t>0$, i.e. $y_t$ 'moves towards' $g_+$. The blue half circle can be identified with $\P(Y_g)$.}\label{fig:hypconvex-proof2}
\end{figure}
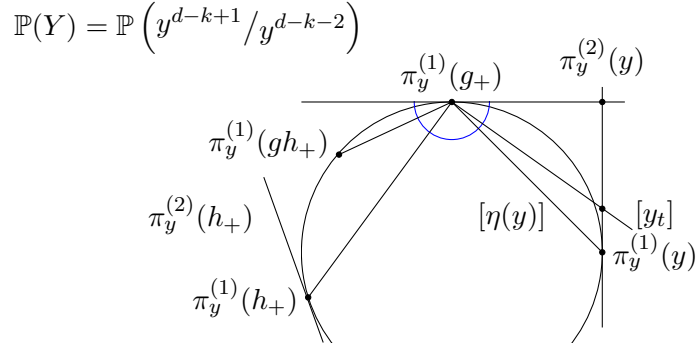
Moreover for the linear approximation $L_y$ of $\xi^{d-k-1}$ at $y$ defined by $ [L_y(t)]_Y= [y_t]_{Y}\cap \pi_y^{(2)}(y)$, it holds 
$$\left.\frac{d}{dt}\right|_{t=0} \gc_{k+1}(h_+^{k+1},y^{d-k-1},L_y(t),g_+^{k+1})>0$$
because  $\r$ satisfies properties $H_{k+1},H_{d-k-1}$ and thus is $(k+1)-$positively ratioed.
This implies that, for $t>0$, the points
$$\left(\pi_y^{(2)}(g_+)\cap \pi_y^{(2)}(y),[L_y(t)],\pi_y^{(1)}(y),\pi_y^{(2)}(h_+)\cap \pi_y^{(2)}(y)\right)$$
 are in that cyclic order on $\P(\pi_y^{(2)}(y))$ (see Corollary \ref{cor.projective and k-cross ratio}). Since $ [L_y(t)]_Y= [y_t]_{Y}\cap \pi_y^{(2)}(y)$, we derive for any sufficiently small positive $t$ that
$$\left([\eta(y)]_{Y_g}=[y^{d-k-1}]_{Y_g},[y_t]_{Y_g},[\pi_y^{(2)}(g_{+})]_{Y_g},[\pi_y^{(1)}(gh_{+})]_{Y_g}, [\pi_y^{(1)}(h_{+})]_{Y_g}\right)$$
are in that cyclic order on $\Pp(Y_g)\simeq\P(\quotient{Y}{ \pi_y^{(1)}(g)})$ - compare Figure \ref{fig:hypconvex-proof2}. This finishes the proof, as we can now derive with Lemma \ref{lem.symmetry of projective cro}$(9)$ that Equation \eqref{e.pcry} holds.
\end{proof}

\subsection*{Step 3: Conclusion}
We know from Corollary \ref{cor.root compared to grassmannian cro} that
$$\begin{array}{l}
\left(1-\frac{\lambda_{k+1}}{\lambda_k}(h_{\rho})\right)^{-1}<
\gc_{d-k}\left(h_-^{d-k},h_+^k, gh_+^{k}, \left(g_+^k\cap h_-^{d-k+1}\right)\oplus h_-^{d-k-1} \right)
\end{array},
$$
and from Proposition \ref{prop.step-2 cross ratio increases} that
$$x\mapsto \gc_{d-k}\left(h_-^{d-k},h_+^k, gh_+^k, \left(g_+^k\cap x^{d-k+1}\right)\oplus x^{d-k-1} \right)$$ 
is non-decreasing for $x\in (h_{-},g_+)_{g_-}$ moving towards $g_+$. Thus
$$ \left(1-\frac{\lambda_{k+1}}{\lambda_k}(h_{\rho})\right)^{-1} < \gc_{d-k}\left(h_-^{d-k},h_+^k, gh_+^{k}, g_+^{d-k}\right).$$
Moreover $\r$ is $(d-k)-$positively ratioed, because it satisfies properties $H_k,H_{d-k}$ (Theorem \ref{prop.hyperconvex implies positively ratioed}). Since  $h_-,g_-,h_+,gh_+,g_+$ are in that cyclic order, this yields via the cocycle identity (Lemma \ref{lem property of grassmannian cro} $(2)$) that
$$ \gc_{d-k}\left(h_-^{d-k},h_+^k, gh_+^{k}, g_+^{d-k}\right)< \gc_{d-k}\left(g_-^{d-k},h_+^k, gh_+^{k}, g_+^{d-k}\right).$$
Theorem \ref{thm.collar lemma} follows then from Corollary \ref{cor.periods of k cross ratio}, stating that 
$$\gc_{d-k}(g_-^{d-k},h_+^k,gh_+^k, g_+^{d-k})=\frac{\l_1\ldots\l_k}{\l_{d-k+1}\ldots\l_d}(g_{\r}).$$
\addtocontents{toc}{\protect\setcounter{tocdepth}{2}}

\section{A counterexample to the strong collar lemma}\label{sec.strong collar}
The goal of the section is to prove Theorem \ref{thm.INTROcounterexample} from the introduction, which we recall for the reader's convenience (here $\G_{1,1}$ is the fundamental group of the once punctured torus):
\begin{thm}
There is a one parameter family of positive representations $\rho_x:\G_{1,1}\to \PSL(3,\R)$, for $x\in(0,\infty)$, and $\g$, $\d\in\G_{1,1}$ such that 
$$\ell_{\alpha_1}(\rho_x(\g))=\ell_{\alpha_1}(\rho_x(\delta))\to 0$$
as $x$ goes to zero.
\end{thm}
\begin{proof}
We realize the once punctured torus as the quotient of a square modulo the identification of parallel sides, in such a way that the puncture is the image of the vertices. We fix the triangulation of such surface whose sides are the positive diagonal, and the two sides of the square. 
Following Fock-Goncharov positive representations of $\G_{1,1}$ are uniquely determined by 6 shear
invariants (corresponding to the three sides of the triangulation) and 2 triple ratios \cite{Fock-Goncharov}. We will set all shear invariants to be zero, while the triple ratios will degenerate (in opposite directions) along the sequence.

Given three flags $([A^1],\langle A^1,A^2\rangle),([B^1],\langle B^1,B^2\rangle),([C^1],\langle C^1,C^2\rangle)\in\calF(\R^3)$ and an identification $\wedge^3\R^3\simeq \R$ their \emph{triple ratio} is defined by
$$\tau(A,B,C)
=\frac{A^1\wedge A^2\wedge B^1}{A^1\wedge A^2\wedge C^1} \cdot\frac{B^1\wedge B^2\wedge C^1}{B^1\wedge B^2\wedge A^1}\cdot\frac{C^1\wedge C^2\wedge A^1}{C^1\wedge C^2\wedge B^1}.$$
It is immediate to check that the value of $\tau(A,B,C)$ doesn't depend on the choices involved.
Moreover given two flags $A,C\in \calF(\R^3)$ and two transverse lines $[B^1],[D^1]\in\P(\R^3)$. The \emph{shears} are defined by (compare \cite[Section 2.6]{Par})%
$$\sigma(A,[B^1],C,[D^1]):=\left(\log (-\pc_{A_1}(A,B_1,D_1,C_1)),\log (-\pc_{C_1}(C,B_1,D_1,A_1))\right).$$
In our example we will assume that all the shears are equal to 0, this corresponds to the points being in harmonic position. 

Observe that given two flags $A,C$ and two lines $[B^1],[D^1]$, the shear has the form $\sigma(A,[B^1],C,[D^1])=(0,0)$ if and only if there exists a basis $e_1,e_2,e_3$ with 
$$\begin{array}{l}
A=([e_1],\langle e_1,e_2\rangle)\\
B_1=[(1,-1,1)^T]\\
C=([e_3],\langle e_3,e_2\rangle)\\
D_1=[(1,1,1)^T]\\
\end{array}
$$

We can then consider the positive representation $\rho_{x}:\G_{1,1}\to \PSL(3,\R)$ of the fundamental group $\G_{1,1}$ of the once punctured torus whose Fock-Goncharov parameters are given by the triangle invariant $x,x^{-1}$ and all whose shears are fixed equal to $(0,0)$.

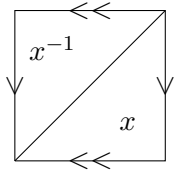
\begin{center}
\begin{tikzpicture}[scale=2]
\draw (0,0) to [red] (0,1) to [blue](1,1) to[red] (1,0) to[blue] (0,0);
\draw (0,0) to (1,1);
\node at (.25,.75) {$x^{-1}$};
\node at (.75,.25) {$x$};
\node at (0,.5) [rotate=90] {$<$};
\node at (1,.5) [rotate=90] {$<$};
\node at (0.5,0) {$<<$};
\node at (0.5,1)  {$<<$};
\end{tikzpicture}
\end{center}
We denote by $\g\in\G_{1,1}$ the element realizing the identification of the vertical sides, and by $\delta\in\G_{1,1}$ the element realizing the identification of the horizontal sides. Up to conjugating the representation we can assume that the two endpoints of the diagonal are associated to the standard flags 
$$\begin{array}{l}
\underline\infty:=([e_1],\langle e_1,e_2\rangle)\\
\underline 0:=([e_3],\langle e_3,e_2\rangle).
\end{array}$$
 In order to compute representatives of $\rho_{x}(\g)$ and $\rho_{x}(\delta)$ we need to compute the flags $\underline{t}$, $\underline{s}$ determined by 
\begin{align*}
\tau(\underline{\infty},\un{s},\un{0})&=x^{-1}\\
\tau(\underline{0},\un{t},\un{\infty})&=x
\end{align*}
As well as the lines in the images $\g(\un\infty^1)$, $\delta(\un\infty^1)$, which are uniquely determined by our requirements on the shears.
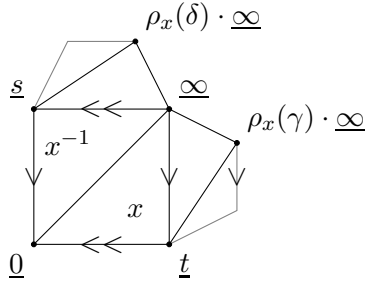
\begin{center}
\begin{tikzpicture}[scale=1.8]
\draw (0,0) to [red] (0,1) to [blue](1,1) to[red] (1,0) to[blue] (0,0);
\draw (0,0) to (1,1);
\node at (.25,.75) {$x^{-1}$};
\node at (.75,.25) {$x$};
\node at (0,.5) [rotate=90] {$<$};
\node at (1,.5) [rotate=90] {$<$};
\node at (0.5,0) {$<<$};
\node at (0.5,1)  {$<<$};
\draw (1,1) to (.75,1.5) to (0,1);
\draw [gray] (0,1) to (.25,1.5) to (.75,1.5);
\draw (1,1) to (1.5,.75) to (1,0);
\draw [gray] (1,0) to (1.5,.25) to (1.5,.75);
\node at (1.5,.5) [rotate=90] {$<$};
\filldraw (0,0) circle [radius=.5pt] node [below left] {$\un 0$};
\filldraw (0,1) circle [radius=.5pt] node [above left] {$\un s$};
\filldraw (1,0) circle [radius=.5pt] node [below right] {$\un t$};
\filldraw (1,1) circle [radius=.5pt] node [above right] {$\un \infty$};
\filldraw (.75,1.5) circle [radius=.5pt] node [above right] {$\rho_{x}(\delta)\cdot\un \infty$};
\filldraw (1.5,.75) circle [radius=.5pt] node [above right] {$\rho_{x}(\g)\cdot\un \infty$};
\end{tikzpicture}
\end{center}
It is easy to check that with our assumptions we have
\begin{align*}
\un t&=\left(\left[\bpm1\\1\\1\epm\right],\left\langle \bpm1\\1\\1\epm, \bpm 1\\0\\-x \epm\right\rangle\right) \quad  &\rho_{x}(\g)\cdot\un \infty^1= \bpm 2x^{-1}+2\\2\\1\epm\\
\un s&=\left(\left[\bpm1\\-1\\1\epm\right],\left\langle \bpm1\\-1\\1\epm, \bpm 1\\0\\-x \epm\right\rangle\right) \quad  &\rho_{x}(\delta)\cdot\un \infty^1= \bpm 2x^{-1}+2\\-2\\1\epm.\\\\
\end{align*}
One directly checks that the matrices for the elements $\rho_{x}(\g),\rho_{x}(\delta)$ are
\begin{align*}
\rho_{x}(\g)&=\sqrt[3]{x^{-1}}\bpm 2x+2&2x+2&1\\ 2x&2x+1&1\\x& x+1&1\epm \\
 \rho_{x}(\delta)&=\sqrt[3]{x^{-1}}\bpm 2x+2&-2x-2&1\\-2x&2x+1&-1\\x&-x-1&1\epm.
\end{align*}

The characteristic polynomials of these two matrices are both given by
\begin{align*}
\chi(\lambda)=\lambda^3-\lambda^2(4x^{-\frac{1}{3}}+4x^{\frac{2}{3}})+\lambda(4x^{\frac{1}{3}}+4x^{-\frac{2}{3}})-1.
\end{align*}

We want to consider the limit as $x\to 0$. To simplify the equations we substitute $y=x^{-\frac{1}{3}}$. Hence we get
\begin{align}
&\l_1(y)+\l_2(y)+\l_3(y)=4(y+y^{-2})\label{eq.1}\\
&\l_1(y)\l_2(y)+\l_1(y)\l_3(y)+\l_2(y)\l_3(y)=4(y^2+y^{-1})\label{eq.2}\\
&\l_1(y)\l_2(y)\l_3(y)=1.\label{eq.3}
\end{align}
The eigenvalues are then necessarily positive (cfr. Proposition \ref{prop.INTROroots are positive for Hk}) and are ordered so that $\l_1(y)\geq \l_2(y)\geq \l_3(y)$. Hence Equation \eqref{eq.1} yields that for $y\to\infty$ we have $\lambda_1(y)\to \infty$ and thus by Equation \eqref{eq.3} $\lambda_3(y)\to 0$. Dividing Equation \eqref{eq.1} by $y$ we see that for every $\epsilon>0$ there exists $N_{\epsilon}\in \R$ such that $\lambda_1(y)\slash y> 2-\epsilon$ for all $y\geq N_{\epsilon}$. 

We claim that $k:=\liminf \lambda_2(y)\slash y>0$: if $\liminf \lambda_2(y)\slash y=0$ Equation \eqref{eq.2} implies, by dividing with $y^2$, that $\limsup \lambda_1(y)\slash y=\infty$. In this case Equation \eqref{eq.1} would yield that $\liminf \lambda_2(y)\slash y=-\infty$; a contradiction. This argument yields also $\limsup \lambda_1(y)\slash y<\infty$, and thus $\limsup \lambda_2(y)\slash y<\infty$. Hence if we pass to an increasing sequence $\{y_n\}\subset\R$ such that the limits $A:=\lim_{n\to\infty} \lambda_1(y_n)\slash y_n$, $B:=\lim_{n\to\infty} \lambda_2(y_n)\slash y_n$ exist, then those limits have to satisfy $A+B=4$ and $AB=4$, i.e. $A=2=B$. In particular it follows that $\lim_{y\to\infty} \lambda_1(y)\slash y=2$, $\lim_{y\to\infty} \lambda_2(y)\slash y=2$. This yields $\lim_{y\to\infty} \lambda_1(y)\slash\lambda_2(y)=1$.

Finally, as $\r_x(\g),\r_x(\delta)$ have the same characteristic polynomial, which satisfies $\lambda_1(x)\slash\lambda_2(x)\to 1$ for $x\to 0$, we get the claim.
\end{proof}
Observe that, as the representation $\rho_x$ has unipotent boundary holonomy it is not restriction of a Hitchin representation of the double of the surface. It is however easy to choose small shears $\sigma(x)$ so that the associated sequence $\rho_x'$ of representations has loxodromic boundary holonomy and can therefore be doubled to a Hitchin representation \cite[Section 9.2]{LabMc}.

\bibliography{mybib}
\bibliographystyle{alpha}

\end{document}